\documentclass{amsart}
\usepackage{amsmath,amsfonts,amssymb,amsthm,esint, graphicx}

\usepackage[colorlinks,citecolor=blue]{hyperref}

\newtheorem{theorem}{Theorem}[section]
\newtheorem{definition}[theorem]{Definition}

\newtheorem{lemma}[theorem]{Lemma}
\newtheorem{proposition}[theorem]{Proposition}
\newtheorem{corollary}[theorem]{Corollary}

\theoremstyle{remark}
\newtheorem{remark}[theorem]{Remark}
\newtheorem{example}[theorem]{Example}

\numberwithin{equation}{section}
\numberwithin{figure}{section}

\newcommand{\R}{\mathbb{R}}
\newcommand{\N}{\mathbb{N}}
\renewcommand{\P}{\mathbb{P}}

\newcommand{\supp}{\operatorname{supp}}
\newcommand{\diam}{\operatorname{diam}}
\newcommand{\Img}{\operatorname{Im}}
\newcommand{\diag}{\operatorname{diag}}
\newcommand{\Span}{\operatorname{span}}

\newcommand{\Id}{\mathrm{Id}}
\newcommand{\fone}{\mathbf{1}}
\newcommand{\vone}{\mathbf{1}}
\newcommand{\eqset}{:=}

\begin{document}
	\title[Signature of squared distance matrices for countable MMS]{On the signature of squared distance matrices of metric measure spaces}

	\author{Alexey Kroshnin}
	\address{
		The Weierstrass Institute for Applied Analysis and Stochastics, 
		Mohrenstraße 39, 10117 Berlin, Germany
	}
	\email{alex.kroshnin@gmail.com}

	\author{Tianyu Ma}
	\address{
		Faculty of Mathematics, HSE University, Usacheva ul. 6, 119048 Moscow, Russian Federation 
	}
	\email{tma@hse.ru}

	\author{Eugene Stepanov}
	\address{
		Dipartimento di Matematica, Universit\`a di Pisa,
		Largo Bruno Pontecorvo 5, 56127 Pisa, Italy
		\and 
		St.Petersburg Branch of the Steklov Mathematical Institute of the Russian Academy of Sciences,
		St.Petersburg, 
		Russian Federation
		\and
		HSE University, Moscow, Russian Federation		
	}
	\email{stepanov.eugene@gmail.com}

	\thanks{	
		E.S.\ acknowledges the MIUR Excellence Department Project awarded to the Department of Mathematics, University of Pisa, CUP I57G22000700001.	
		The work is also partially within the framework of HSE University Basic Research Program.
		}
	
	\dedicatory{Dedicated to the memory of Anatoly Moiseevich Vershik}
	
	\begin{abstract}
		We consider the numbers of positive and negative eigenvalues of matrices of squared distances between randomly sampled i.i.d.\ points in a given metric measure space. 
		These numbers and their limits, as the number of points grows, in fact contain some important information about the whole space. 
		In particular, by knowing them, we can determine whether this space can be isometrically embedded in the Hilbert space.
		We show that the limits of these numbers exist almost surely, are nonrandom and the same for all Borel probability measures of full support, and, moreover, are naturally related to the operators defining the multidimensional scaling (MDS) method. 
		We also relate them to the signature of the pseudo-Euclidean space in which the given metric space can be isometrically embedded. 
		In addition, we provide several examples of explicit calculations or just estimates of those limits for sample metric spaces. 
		In particular, for a large class of countable spaces (for instance, containing all graphs with bounded intrinsic metrics), we get that the number of negative eigenvalues increases to infinity as the size of samples grows. 
		However, we are able to provide examples when the number of samples grows to infinity and the numbers of both negative and positive eigenvalues increases to infinity, or the number of positive eigenvalues is bounded (but as large as desired), and the number of positive ones is fixed.
		Finally, we consider the example of the universal countable Rado--Erd{\H o}s--R{\'e}nyi graph. 	
	\end{abstract}
	
	\maketitle

	\section{Introduction}
	
	Let $(X,d)$ be a metric space (i.e., $X$ a nonempty set equipped with a distance $d$) and $\mu$ be a Borel probability measure on $X$.
	We further assume, without loss of generality, that $(X,d)$ is separable. 
	The so-called \textit{learning} problem for the metric measure space $(X,d,\mu)$ is to recover the information on the
	triple $(X,d,\mu)$ from the information on distances between points of an appropriately chosen subset of $X$.
	In view of the famous result of M.~Gromov (so-called ``\textit{mm-reconstruction theorem}''~\cite[section $3.\frac{1}{2}.5$]{Gromov-metricRiem99}) 
	further reproved and generalized by A.~Vershik in~\cite{Vershik-learningMMS-UMN98} and~\cite{Vershik-randomMMS04a}, this problem is completely solved. 
	A Polish (i.e.,\ separable and homeomorphic to complete) space $(X,d,\mu)$ can be reconstructed from distances between points obtained from the following randomized procedure:
	for every $m\in \N$, one chooses i.i.d.\ random points $\xi_i \in X$ drawn from the distribution $\mu$, $i=1,\ldots, m$, thus forming a finite random sequence of points $X_m \eqset (\xi_i)_{i=1}^m$, and calculates 
	the $m\times m$ random distance matrix $(d(\xi_i,\xi_j))_{i,j=1}^m$. The joint distribution law of all these matrices (i.e.,\ for all $m\in \N$) then determines $(X,d,\mu)$ uniquely up to a measure-preserving isometry. To reconstruct the space $(X,d,\mu)$, this approach requires the knowledge of the statistics of random distance matrices of all sizes, which is quite difficult to accomplish in practice. The natural question is to what extent one may limit himself to the statistics of something less than the whole set of distance matrices, e.g., to the spectra of the latter, or even to some function of the spectra (e.g., the numbers of positive and negative eigenvalues). In other words, what kind of information on $(X,d,\mu)$ is encoded by the spectra of random distance matrices, or even by some functions of the spectra. Several attempts to answer these questions have been made, even numerically in~\cite{BogomolnyBohigasSchmit-disteigen08}, and for some functions of Euclidean distances in~\cite{Bordenave-EuclDistMatr08,Zeng-disteigen14,Zeng-disteigen22,Jiang-disteigen15}. However, they still remain very far from even being partially understood.
	
	In this paper, we mainly concentrate on information on signatures. More precisely, we focus on the numbers of positive and negative eigenvalues
	of random \textit{squared} distance matrices $(d^2(\xi_i,\xi_j))_{i,j=1}^m$, as $m\to\infty$. The methods in the above-cited Gromov--Vershik result on distance matrices can be applied to our research on squared distances at little cost: the joint distribution law of squared distance
	(or any monotone function of a distance) matrices still determines $(X,d,\mu)$ uniquely up to a measure-preserving isometry. However, working with $d^2$ instead of $d$ looks much more natural precisely in terms of the relationship between the spectra of the respective matrices and the properties of these metric measure spaces. For instance, it follows from the Schoenberg theorem~\cite[theorem 3.1]{alfakih2018euclidean} that all matrices
	$(d^2(x_i,x_j))_{i,j=1}^m$ with $x_i\in X$, $m\in \N$ are conditionally negative semidefinite (i.e.,\ negative semidefinite on the hyperplane $x_1+\ldots+x_m=0$) 
	if and only if $(X,d)$ is isometrically embeddable in a Hilbert space. 
	In this case, when $X$ is finite, an embedding can be explicitly reconstructed using the \textit{multidimensional scaling} (MDS) algorithm well known in data science~\cite{wang2012geometric}. 
	Thus, we prefer to state our results in terms of the signatures of the matrices 
	\[
	S(X_m) \eqset -\frac{1}{2} \left(d^2(\xi_i,\xi_j)\right)_{i,j=1}^m. 
	\]
	The purely ``aesthetic'' coefficient $-1/2$ here is only to reflect the fact that exactly these matrices are involved directly in the MDS algorithm.
	
	We will show that 
	the random numbers $s_+(S(X_m))$ of positive and $s_-(S(X_m))$ of negative random eigenvalues of 
	matrices $S(X_m)$, 
	almost surely converge to some fixed numbers $s_+$ and $s_-$ respectively, and
	provide several ways to compute them.
	In particular, we show that under a rather mild condition on $(X,d,\mu)$, the numbers 
	$s_+$ and $s_-$ count the numbers of positive and negative eigenvalues of the operator $K\colon L^2(X,\mu)\to L^2(X,\mu)$ defined by the formula
	\[
	(Ku)(x)\eqset -\frac{1}{2}\int_X d^2(x,y) u(u)\, d\mu(y).
	\]
	Therefore, we write in this case $s_\pm=s_\pm(K)$.
	This is true when $\mu$ has finite $4$-th order moment, i.e.,\ 
	\[
	\int_X d^4(x_0,y) \, d\mu(y)<+\infty
	\]
	for some (and hence for all) $x_0\in X$, in particular, when $(X,d)$ is bounded.
	Under this condition, the spectra of $S(X_m)$ 
	converge almost surely as $m\to\infty$ in many reasonable senses to the spectrum of $K$~\cite{KoltchinskiiGine-RandomIntOp2000}. The operator $K$ 
	is involved in the construction of MDS for general metric measure spaces $(X,d,\mu)$ as the limit of MDS for finite samples from $X$ when the number of samples grows. We will show further that 
	\begin{itemize}
		\item[(i)] $s_+$ and $s_-$ do not depend on $\mu$ if the latter has full support, i.e.,\ $\supp\mu=X$;
		\item[(ii)] The number $s_+$ may be viewed as a natural lower bound for the dimension of the MDS embedding of $(X,d,\mu)$. 
		More generally, the couple of numbers $(s_-, s_+)$ give the lower bounds for the signature of a pseudo-Euclidean space in which $(X,d)$ can be isometrically embedded (the respective embedding can be constructed explicitly). These bounds are in fact almost tight in the sense
		that $(X,d,\mu)$ cannot be isometrically embedded in a pseudo-Euclidean space with signature $(\sigma_-,\sigma_+)$ 
		with $\sigma_\pm < s_\pm -1$.
		As an easy corollary, we get that every finite metric space can be isometrically embedded in a pseudo-Euclidean one;
		\item[(iii)] one has that $s_+=+\infty$ when $(X,d)$ is bounded, $X$ is countable 
		and $d$ is separated from zero, e.g., for
		countable graphs equipped with bounded intrinsic distance. As a corollary, the embedding of $(X,d,\mu)$ produced by MDS is not finite-dimensional.
	\end{itemize}
	We also
	\begin{itemize}
		\item[(iv)] provide several examples of calculations of $s_\pm$ for particular metric spaces. 
		In many interesting cases of infinite metric spaces, we get $s_+=s_-=+\infty$. 
		\item[(v)]
		However, we provide examples of metric spaces with arbitrarily large finite $s_-$. It seems easier to provide examples 
		when both $s_+$ and $s_-$ are large. This corresponds well to the observation of~\cite{Charles2013} about so-called
		hollow symmetric non-negative (HSN) matrices (symmetric matrices with nonnegative entries and with zeros on the diagonal)
		that they ``normally'' have bigger number of negative rather than positive eigenvalues (recall that squared distance matrices are a particular case of HSN matrices). 
		\item[(vi)]
		Nevertheless, we are able to show a construction of a finite metric space with $s_-$ as large as desired and $s_+$ fixed
		(up to an error of at most $1$). In particular, this provides a construction of HSN matrices (even stronger, squared distance matrices)
		with arbitrarily large number of positive eigenvalues which is alternative to a seemingly much heavier construction of~\cite{Charles2013}.
		\item[(vii)] We also provide some natural estimates on $s_\pm$.
	\end{itemize}
	Finally,
	\begin{itemize}
		\item[(viii)] in the quite natural example of the Rado--Erd{\H o}s--R{\'e}nyi graph (which is a universal graph containing all the finite graphs as its proper subgraphs), one has $s_+=s_-=\infty$;
		\item[(ix)] however, the limits of the ratios $s_+(S(X_m))/s_-(S(X_m))$ as $m\to\infty$ may depend on $\mu$ as illustrated by the example of the Rado--Erd{\H o}s--R{\'e}nyi graph.
	\end{itemize}
	

	\section{Notation and preliminaries}
	
	\subsection{Notation}
	The metric measure space $(X,d,\mu)$ is always assumed to be separable with
	$\mu$ a Borel probability measure. 
	
	For a set $D$, we denote by $\fone_D$ its characteristic function. 
	If $D$ is a subset of a metric space, we let $\bar D$ stand for its closure.
	For a compact linear operator $K$ over some normed space, we denote by $s_\pm(K)$ the numbers of its positive and negative eigenvalues, respectively. In this paper, unless otherwise mentioned explicitly, the numbers of eigenvalues are always counted with multiplicity. 
	We write $\supp \mu$ for the support of a measure $\mu$.
	
	For a Hilbert space $H$ we denote by $(\cdot,\cdot)$ and $\|\cdot\|$ the respective scalar product and the Hilbert norm. 
	The same notation will stand for the (indefinite) scalar product and the norm (sometimes called ``interval'' in physics ) in a pseudo-Euclidean space. 
	For a linear operator $A$ between two Banach spaces, we denote by $\Img A$ its range, by $\ker A$ its kernel (i.e.\ the zero space) and by $A^*$ its adjoint. If $L$ is a subset of a linear space, then $\Span L$ stands for its linear span.
	
	By $L^p(X,\mu)$ we denote the usual Lebesgue space of
	integrable (with respect to $\mu$) with power $p\geq 1$ (or $\mu$-essentially bounded when $p=+\infty$) functions on $X$, the canonical norms in these spaces being denoted by $\|\cdot\|_p$. 
	Denote by $\fone^{\perp}$ the orthogonal complement of constant function $\fone$ in $L^2(X,\mu)$, and by $P_\mu$ the orthogonal projector from $L^2(X,\mu)$ to $\fone^{\perp}$. 
	Let
	\begin{align}
		\label{eqn:MDS opK}
		&(K_\mu)(x)\eqset -\dfrac{1}{2}\int_X d^2(x,y) u(y) d\mu(y)\\
		\label{eqn:MDS opT}	
		&T_\mu\eqset (P_\mu K_\mu P_\mu)(u),
	\end{align}
	be linear operators defined on $L^2(X,\mu)$.
	We usually avoid the subscript $\mu$ and write the above introduced operators as $P, K, T$ instead of
	$P_\mu, K_\mu, T_\mu$ respectively, unless we would like to emphasize their dependence on the measure $\mu$.
	
	If the $4$-th moment of $\mu$ is infinite, then the domains of $K$ and $T$ are not
	all of $L^2(X,\mu)$. In the opposite case, both $K$ and $T$ are self-adjoint Hilbert-Schmidt operators over the whole $L^2(X,\mu)$. 
	Throughout this paper, we call the operator $K$ and $T$ the MDS defining operators associated with the metric measure space $(X,d,\mu)$. The infinite MDS maps of $(X,d,\mu)$ are constructed from the positive eigenvalues and their corresponding $L^2(X,\mu)$-normalized eigenfunctions of $T$.
	Note that
	\begin{equation}\label{eq_defT1}
		\begin{aligned}
			(T f) (x) &= \int_X k_T(x, y) f(y) \,d \mu(y), \quad\mbox{where}\\
			k_T(x,y) &\eqset k(x, y) - \int_X k(x, y') \,d \mu(y') - \int_X k(x', y) \,d \mu(x') \\
			& \qquad \qquad
			+ \int_X \int_X k(x', y') \,d \mu(x') \,d \mu(y').
		\end{aligned}
	\end{equation}
	where for brevity we denoted $k(x,y)\eqset -d^2(x,y)/2$.
	In particular,~\eqref{eq_defT1} implies
	\begin{equation}\label{eq_defT2}
		k_T(x,x)+k_T(y,y)-2 k_T(x,y)= d^2(x,y).
	\end{equation}

	For the general metric measure space $(X,d,\mu)$ and $X_N\eqset \{x_i\}_{i=1}^N\subset X$ a finite \textit{subset}, we denote
	$S(X_N)\eqset \left(-\frac{1}{2}d^2(x_i,x_j)\right)_{i,j=1}^N$. Note that for a finite \textit{sequence}
	of i.i.d.\ random elements $\xi_i$ of $X$, we write $X_m\eqset (\xi_i)_{i=1}^m\subset X$ in lower-case index for the finite random sequence 
	and $S(X_m)\eqset \left(-\frac{1}{2}d^2(\xi_i,\xi_j)\right)_{i,j=1}^m$ for the respective random $m\times m$ matrix.
	It is important to note that $X_m$ does not necessarily always have exactly $m$ elements because some of $\xi_i$ may occasionally coincide. 
	Therefore, $S(X_m)$ can be different from $S(\hat{X}_N)$, where $\hat{X}_N$ is the \textit{set} of distinct elements of the \textit{sequence} $X_m$.\footnote{In what follows we maintain the following notation: if $X_m$ is a finite sequence with possibly repeating elements, then $\hat X_N$ will stand for the set of its elements, or, equivalently, the sequence obtained from $X_m$ by cancelling repetitions. However, if the sequence (random or deterministic) $X_N$ does not contain repetitions, i.e.,\ can be viewed as a set, then we do not write $\hat X_N$ in this case.} 
	
	We denote by $\vone_n$ the vector $\vone_n\eqset (1,\ldots, 1)^T\subset \R^n$ and by $\Id_n$ the identity $n\times n$ matrix
	(the superscript $T$ in this context stands for the transpose of matrices), and $\Pi_n \eqset \Id_n - \frac{\vone_n \vone_n^T}{n} \in \R^{n \times n}$.
	
	\subsection{Spectral theory for matrices}

	In the sequel, we frequently use the Cauchy interlacing theorem~\cite[theorem 17.17]{Mckee2021}. We state it here for convenience in the form we need in the sequel.
	
	\begin{proposition}[Cauchy interlacing theorem]
		Let $B$ be a real 
		symmetric
		matrix of order $N$. Let $\lambda_1\leq \lambda_2\leq \dots\leq \lambda_N$ be the eigenvalues of $B$ counting multiplicity. For any principal submatrix $B'$ of $B$ of order $N-1$, the eigenvalues $\lambda'_1\leq \lambda'_2\leq \dots\leq \lambda'_{N-1}$ of $B'$ are interlacing such that
		\[
		\lambda_i\leq \lambda'_i\leq \lambda_{i+1}\ \quad\mbox{for all } 1\leq i\leq N-1.
		\]
		In particular, for any real $\delta$, if the symmetric matrix $B$ has a principal block admitting at least $m$ eigenvalues greater (resp., less) than or equal to $\delta$, so does the matrix $B$.
	\end{proposition}
	
	\section{Gromov--Vershik scheme}
	
	For a general metric measure space $(X,d,\mu)$, 
	we refer to the experiment consisting of choosing i.i.d.\ random points $\xi_i \sim \mu$ for $i=1,\ldots, m$, and calculating the random matrix $S(X_m)$ where $X_m$ is the finite random sequence of points $X_m\eqset (\xi_i)_{i=1}^m\subset X$, as the \textit{Gromov--Vershik scheme}. 
	Namely, let $(\Omega, \Sigma, \P)$ be a probability space and 
	$\xi_i\colon \Omega\rightarrow X$, $i\in \N$, be i.i.d.\ random elements of $X$ such that $\operatorname{law}(\xi_i)=\mu$ for all $i\in \N$. We set then $X_m(\omega)\eqset (\xi_i(\omega))_{i=1}^m$ for all $\omega\in\Omega$. 
	Note that $X_m$ is not necessarily a random \textit{set}, but rather a random finite sequence of elements of $X$, some of which may be repeating. cancelling the repeating elements of
	the latter, we obtain a random set $\hat X_N(\omega)\subset X$ with $N=N(m,\omega)\eqset \# X_m(\omega)$. We denote its elements by
	$\hat\xi_1(\omega),\ldots, \hat\xi_N(\omega)$, i.e., $\hat X_N(\omega)\eqset \{\hat \xi_i(\omega)\}_{i=1}^N$. We will then say that
	the random set $\hat X_N$ and the respective random matrix $S(\hat X_N)$ are obtained by \textit{Gromov--Vershik scheme without repetitions}.
	
	The random matrix $S(X_m)$ may be viewed as the $m\times m$ principal submatrix of the infinite random matrix
	$\left(-\frac{1}{2}d^2(\xi_i,\xi_j)\right)_{i,j=1}^\infty$. The distribution law of the latter is clearly 
	the measure $\nu\eqset f_{\#}\mu^\infty$ (which may be called \textit{matrix distribution} following~\cite{Vershik-learningMMS-UMN98,Vershik-randomMMS04a}) over the space of infinite matrices $\R^{\N\times \N}$, where $\mu^\infty$ is the countable tensor product of $\mu$,
	$f\colon X^\N\to \R^{\N\times \N}$ is defined by
	\[
	f((x_1, x_2, \ldots, x_k,\ldots)) \eqset \left(-\frac{1}{2}d^2(x_i,x_j)\right)_{i,j=1}^\infty,
	\]
	and $f_{\#}$ stands for the usual pushforward of a measure by the map $f$.
	
	We now reformulate the Gromov--Vershik theorem (``\textit{mm-reconstruction theorem}'', see~\cite{Vershik-learningMMS-UMN98} or alternatively~\cite{Vershik-randomMMS04a})
	which in its original version uses the random distance matrix $(d(\xi_i,\xi_j))_{i,j=1}^\infty$ instead of $\left(-\frac{1}{2}d^2(\xi_i,\xi_j)\right)_{i,j=1}^\infty$.
	
	\begin{theorem}
		If $(X,d)$ is Polish, then the measure $\nu$ determines $(X,d,\mu)$ uniquely up to a measure-preserving isometry.
	\end{theorem}

	\section{Limit distance signature of separable metric spaces}
	
	We introduce the following notion crucial for this paper.
	
	\begin{definition}\label{def_limsig1}
		Let $(X,d)$ be a metric space.
		We call the limit distance signature of $(X,d)$ the 
		couple $(s_-(X,d), s_+(X,d))$ with $s_\pm(X,d))\in \N\cup\{+\infty\}$ defined by
		\begin{equation}\label{eq_spmSup1}
			s_\pm (X,d) =\sup\{s_\pm (S(X_N)) : X_N\subset X\, \mbox{finite} \}.
		\end{equation}
	\end{definition}
	
	In this section we prove the following result.
	
	\begin{theorem}\label{th_spm1} Let $(X,d)$ be a separable metric space. Then 
		the following assertions hold.
		\begin{itemize}
			\item[(i)] If $X$ is at most countable, then 
			\[
			s_\pm(X,d)=\lim_N s_\pm (S(X_N))=\sup_N s_\pm (S(X_N))
			\] 
			for an arbitrary sequence of finite subsets $X_N\nearrow X$ as $N\to \infty$. In particular, if $X$ is finite, then $s_\pm(X,d)=s_\pm(S(X))$.
			\item[(ii)] If $X$ is infinite, then $s_\pm(X,d)=s_\pm(\tilde X,d)$, where $\tilde{X}\subset X$ is an arbitrary countable dense subset of $X$ endowed with the same distance $d$.
			\item[(iii)] If $\mu$ is a Borel probability measure over $X$ with full support, i.e.,\ with $\supp\mu = X$, then
			\[
			s_\pm(X,d)=\lim_m s_\pm (S(X_m(\omega)))=\lim_m s_\pm (S(\hat X_N(\omega))) \quad\mbox{a.s.},
			\]
			where $X_m$ are random finite sequences of i.i.d.\ points in $X$ chosen according to the Gromov--Vershik scheme and $ \hat X_N$ are the respective random finite sets of points in $X$ obtained by cancelling repeated elements in the sequence $X_m$, i.e.,\ according to the Gromov--Vershik scheme without repetitions.
			\item[(iv)] If $\mu$ has the finite $4$-th moment and full support in $X$, then 
			\[
			s_\pm (K_\mu)=s_\pm (X,d),
			\]
			where $K_\mu$ is the operator defined by~\eqref{eqn:MDS opK}, although the numbers $s_\pm(X,d)$ do not depend on $\mu$.
			Furthermore,
			\[
			s_\pm(K_\mu) - 1 \le s_\pm(T_\mu) \le s_\pm(K_\mu),
			\]
			and if $\fone \not\perp \ker K_\mu$, then $s_\pm(T_\mu) = s_\pm(K_\mu)$, where $T_\mu$ is the operator defined by~\eqref{eqn:MDS opT}.
			\item[(v)] $s_\pm (T_\mu)$ also does not depend on $\mu$ once the latter
			has finite $4$-th moment and full support in $X$.
		\end{itemize}
	\end{theorem}
	
	\begin{remark}
		Note that $s_\pm(X,d)\geq 1$ unless $X$ is a singleton (in which case clearly $s_\pm(X,d)= 0$ by~(i)).
		In fact, for every finite $X_N\subset X$
		one has $s_-(S(X_N))\geq 1$ since the Perron--Frobenius theorem guarantees the existence of a strictly negative eigenvalue
		of maximum absolute value among the other eigenvalues. Moreover, the trace of the above matrices is zero, so we have
		$s_+(S(X_N))\geq 1$.
	\end{remark}
	
	\begin{proof}
		Claim~(i) is Proposition~\ref{prop:sign convergence1}.
		Claim~(ii) is in fact the definition of limit distance signature, the correctness of which is Proposition~\ref{prop_count_to_sep1}. 
		Claim~(iii) is Proposition~\ref{prop:sign convergence3}.
		Claim~(iv) follows from Proposition~\ref{prop:sign convergence2_sep} and Proposition~\ref{prop: countable space sign}. 
		Finally, Claim~(v) is part of Proposition\ref{prop_sigTmu1}.
	\end{proof}
	
	The subsections below are dedicated to the detailed proof of the above Theorem~\ref{th_spm1}.
	
	\subsection{Deterministic characterization of limit distance signatures}
	\label{sec:determ op countable}
	
	The following purely deterministic statement is valid.
	
	\begin{proposition}
		\label{prop:sign convergence1}
		Let $(X,d)$ be an at most countable metric space. 
		Then for all sequences of finite subsets $X_N\subset X$ such that $X_N\nearrow X$ as $N\to \infty$,
		the limits always satisfy
		\[
		s_\pm (X,d)=\lim_N s_\pm (S(X_N)) .
		\]
		In particular, the above limits are independent of the sequence $(X_N)_N$. 
	\end{proposition}
	
	\begin{proof}
		By the Cauchy interlacing theorem, the sequences $\bigl(s_\pm (S(X_N))\bigr)_N$ are non-decreasing
		and hence admit a finite or infinite limit. Moreover, by the definition of $s_\pm(X,d)$, we have
		$s_\pm (X,d)\geq s_\pm (S(X_N))$ for every $N$, and hence
		\[
		s_\pm (X,d)\geq \lim_N s_\pm (S(X_N)) . 
		\] 
		To show the other side of the inequality, note for any finite subset $X'\subset X$, there is an $M\in \N$ such that $X'\subset X_M$. Thus, by 
		the Cauchy interlacing theorem, $s_\pm (S(X'))\leq s_\pm (S(X_M))$. Since $X'$ is arbitrary, we also have the reverse inequality
		\[
		s_\pm (X,d)\leq s_\pm (S(X_M)) \leq \sup_N s_\pm (S(X_N)) =\lim_N s_\pm (S(X_N)) ,
		\]
		concluding the proof. 
	\end{proof}
	
	We now consider the limit distance signature of a general separable metric spaces.
	
	\begin{proposition}\label{prop_dsign_upper1}
		Let $(X,d_X)$ be isometrically embeddable in $(Y,d_Y)$, both metric spaces being separable. 
		Then
		\[
		s_\pm(X,d_X)\leq s_\pm(Y,d_Y).
		\]
	\end{proposition}
	
	\begin{proof}
		By embedding $X$ isometrically into $Y$ we may assume without loss of generality that $X$ is a subset of $Y$. The assertion is then straightforward from the definition of $s_\pm$.
	\end{proof}
	
	\begin{proposition}\label{prop_count_to_sep1} 
		Let $(X,d)$ be a separable metric space. Then 
		\[
		s_\pm(X,d)=s_\pm(\tilde X,d),
		\]
		where $\tilde X\subset X$ is an arbitrary countable 
		dense subset of $X$.
	\end{proposition}
	
	\begin{proof}
		Consider an arbitrary finite set $X_N\subset X$. Since $\tilde X$ is dense in $X$, we can find a finite $\tilde X_N\subset \tilde X$
		close enough to $X_N$ so that $s_\pm (S(X_N))\leq s_\pm (S(\tilde X_N))$. This implies 
		$s_\pm(X,d)\leq s_\pm(\tilde X,d)$ and the reverse inequality follows immediately from the definition of $s_\pm (X,d)$
		(alternatively, from Proposition~\ref{prop_dsign_upper1}) concluding the proof.	
	\end{proof}

	\subsection{Random characterization of limit distance signature}
	\label{sec:rnd op countable}
	
	Consider now 
	\begin{itemize}
		\item the random finite sequences $X_m$ and the respective random matrices $S(X_m)$ obtained by the Gromov--Vershik scheme,
		\item as well as the random finite sets $\hat X_N\subset X$ with $N\le m$ random which is obtained from $X_m$ by cancelling repeating elements,
		and respective random matrices $S(\hat X_N)$ (in other words,
		the random set $\hat X_N$ and the matrix $S(\hat X_N)$ obtained by Gromov--Vershik scheme without repetitions).
	\end{itemize}
	
	The following assertion is valid.
	
	\begin{proposition}
		\label{prop:sign convergence3}
		Let $(X,d)$ be a separable metric space.
		If $\mu$ has full support, one has
		\begin{align*}
			s_\pm(X,d)= \lim_{m} s_\pm(S(\hat X_{N(m,\omega)}(\omega))) = \lim_{m} s_\pm(S(X_m(\omega))) 
		\end{align*}
		for $\P$-a.e.\ $\omega\in\Omega$.
	\end{proposition}
	
	\begin{proof}
		Since by Lemma~\ref{lm_fullsupp1}, one has
		\[
		\overline{\bigcup_m \hat X_{N(m,\omega)}(\omega)} = \supp\mu =X
		\]
		for $\P$-a.e.\ $\omega\in\Omega$, i.e.,\ $\bigcup_m \hat X_N$ is a countable dense set in $X$ almost surely. Therefore, we have
		\[
		s_\pm(X,d)= \lim_{m} s_\pm(S(\hat X_{N(m,\omega)}(\omega)))
		\]
		$\P$-a.s.\ in view of Proposition~\ref{prop_count_to_sep1}.
		By Lemma~\ref{lm_cancelrep_matrix1} below, we also get
		\[
		s_\pm(S(X_m(\omega)))=
		s_\pm(S(\hat X_{N(m,\omega)}(\omega)))
		\]
		for all $\omega\in\Omega$, concluding the proof.
	\end{proof}
	
	The following lemma has been used in the proof of the above Proposition~\ref{prop:sign convergence3}. 
	
	\begin{lemma}\label{lm_fullsupp1}
		Let $(X,d)$ be a separable metric space, and $\mu$ be a Borel probability measure over $X$. 
		Then 
		\[
		\overline{\bigcup_m \hat X_{N(m,\omega)}(\omega)}=\supp\mu 
		\]
		for $\P$-a.e.\ $\omega\in\Omega$, i.e.,\ $\bigcup_m \hat X_N$ is a countable dense set in $\supp\mu$. 
		In particular, if $\supp\mu$ is infinite, then $\lim_m N(m,\omega)=+\infty$ for $\P$-a.e.\ $\omega\in \Omega$ and if additionally
		$\supp\mu$ is countable, then also
		$\hat X_N\nearrow \supp\mu$ a.s.\ as $m\to \infty$.
	\end{lemma}
	
	\begin{proof}
		To prove the first claim, consider an arbitrary open set $U\subset X$ satisfying
		\[
		U\bigcap\overline{\bigcup_m \hat X_{N(m,\omega)}(\omega)}=\emptyset,
		\]
		hence $\xi_i(\omega)\in X\setminus U$ for all $i\in\N$. The probability of this event does not exceed
		$\mu(X\setminus U)^n$ for all $n\in\N$ and hence is zero unless $\mu(U)=0$, i.e.\ $U\cap \supp\mu=\emptyset$, showing 
		that 
		\[ 
		\supp\mu\subset \overline{\bigcup_m \hat X_{N(m,\omega)}(\omega)}=\emptyset.
		\]
		On the other hand, $\P(\{\xi_i\in \supp\mu\})= \mu(\supp\mu)=1$, and hence
		\[ 
		\supp\mu\supset \overline{\bigcup_m \hat X_{N(m,\omega)}(\omega)}
		\]
		a.s., showing the claim.

		To prove the second claim, denote for every $k\in \N$ the set
		\[
		\Omega_k\eqset \{\omega\in\Omega : \lim_m N(m,\omega)\leq k\}
		\] 
		Observe that for all $\omega\in \Omega_k$, one has 
		\[
		\# \bigcup_m \hat X_{N(m,\omega)}(\omega)\leq k.
		\]
		Hence, $\overline{\bigcup_m \hat X_{N(m,\omega)}(\omega)}\neq \supp\mu$ when $\supp\mu$ is infinite. The first claim just proven implies
		$\P(\Omega_k)=0$ in this case. Thus, 
		\[
		\P\left(\left\{\lim_m N(m,\omega)<+\infty \right\}\right) = \P\Bigl(\Omega\setminus \bigcup_k\Omega_k\Bigr)=1
		\] 
		as claimed. 
	\end{proof}
	
	\begin{lemma}\label{lm_cancelrep_matrix1}
		Let the symmetric matrix $\hat S\in \R{^{p\times p}}$ be obtained from the symmetric matrix $S\in \R^{q\times q}$ 
		cancelling all the repeating rows and all the repeating columns.
		Then $s_\pm(S)=s_\pm(\hat S)$.
	\end{lemma}
	
	\begin{proof}
		For each $k$ and $j$, interchanging the $k$-th with $j$-th row and the $k$-th with $j$-th column simultaneously does not change the signature of a matrix, so the proof reduces just to an inductive application of Lemma~\ref{lm_cancelrep_matrix2} below.
	\end{proof}

	\begin{lemma}\label{lm_cancelrep_matrix2}
		Let the symmetric matrix $S\in \R{^{(p+1)\times (p+1)}}$ be of the block form
		\begin{align*}
			S=\begin{bmatrix}
				\hat{S} && v\\
				v^t && a
			\end{bmatrix},
		\end{align*}
		where $\hat S\in \R{^{p\times p}}$ is a symmetric matrix, $v\in \R^p$ and $a\in \R$
		be such that the $p$-th and $(p+1)$-th row as well as the $p$-th and $(p+1)$-th column of $S$ are identical.
		Then $s_\pm(S)=s_\pm(\hat S)$.
	\end{lemma}
	
	\begin{proof}
		Clearly, we have (e.g., by the Cauchy interlace theorem) that $s_\pm(S)\geq s_\pm(\hat S)$. On the other hand, since
		the $p$-th and $(p+1)$-th row as well as the $p$-th and $(p+1)$-th column of $S$ are identical, we have that 
		$s_0(S)= s_0(\hat S)+1$, where $s_0(S)$ stands for the number of zero eigenvalues of $S$ (i.e.\ the dimension of the kernel of the latter). Since also
		\[
		s_-(S)+ s_0(S) + s_+(S) = p+1 , \qquad s_-(\hat S)+ s_0(\hat S) + s_+(\hat S) =p,
		\]
		we get
		\begin{align*}
			s_-( S)+ s_0(S) + s_+(S) &= (s_-(\hat S)+ s_0(\hat S) + s_+(\hat S))+1 =\\
			&= s_-(\hat S)+ s_0(S) + s_+(\hat S) ,
		\end{align*}	
		which implies $s_-(S)+ s_+(S)=s_-(\hat S)+ s_+(\hat S)$, hence the claim.
	\end{proof}

	\subsection{Limit distance signature of the space and signatures of MDS defining operators}
	
	The following result relates the limit distance signature of $(X,d,\mu)$ to the numbers of negative and positive eigenvalues of any
	MDS defining operator.
	
	\begin{proposition}
		\label{prop:sign convergence2_sep}
		Let $(X,d)$ be a separable metric space.
		If $\mu$ has finite $4$-th moment and full support in $X$, one has
		\[
		s_\pm(K_\mu)= s_\pm(X,d),
		\]
		where $K_\mu$ is defined by~\eqref{eqn:MDS opK}.
	\end{proposition}
	
	\begin{remark}
		When $\mu$ has finite $4$-th moment but does not have full support in $X$, we only can assert that
		\begin{equation}\label{eq_sdleqsk1}
			\begin{aligned}
				s_\pm(K_\mu)\leq s_\pm(X,d).
			\end{aligned}
		\end{equation}
		In fact, by the above Proposition~\ref{prop:sign convergence2_sep} one has 
		\[
		s_\pm(K_\mu)=s_\pm(\supp\mu,d),
		\]
		and $s_\pm(\supp\mu,d)\leq s_\pm(X,d)$ by Proposition~\ref{prop_dsign_upper1}.
		The inequality in~\eqref{eq_sdleqsk1} can be sharp as easily seen from the example, say, of $X$ just a Euclidean plane (i.e.,\ $X=\R^2$
		equipped with the Euclidean distance), so that $s_+(X,d)=2$. and $\mu$ having a line as a support, so that
		$s_+(K_\mu)=s_+(\supp\mu,d)=s_+(\R,d)=1$, see Example~\ref{ex_spmEucl1}.
	\end{remark}
	
	\begin{proof}
		Without loss of generality, we may assume $(X,d)$ to be complete (if not, 
		we may extend $d$ and $\mu$ to the completion of $(X,d)$ and work in the latter completion). 
		Then by the estimate~\eqref{eq_spmA2} from Lemma~\ref{lm_approxLinOp1}(ii)
		(applied with $A\eqset K_\mu$, $a\eqset -d^2/2$), one has for some sequence of finite subsets $X_N\subset X$, the following chain of inequalities
		\begin{equation}\label{eq_sKmu1}
			s_+(K_\mu)\leq \liminf_N s_+(S(X_N)) \leq s_+(X,d).
		\end{equation}
		On the other hand,
		\begin{equation*}\label{eq_sKmu2a}
			s_+(K_\mu)\geq s_+(S(X_N))
		\end{equation*}
		for every finite $X_N\subset X$ by estimate~\eqref{eq_spmA1} in Lemma~\ref{lm_approxLinOp1}(i), applied with
		$A\eqset K_\mu$, $a\eqset -d^2/2$. Hence
		\begin{equation*}\label{eq_sKmu2b}
			s_+(K_\mu)\geq s_+(X,d).
		\end{equation*}
		This inequality together with~\eqref{eq_sKmu1} gives the claim
		\[
		s_+(K_\mu)= s_+(X,d).
		\]
		The analogous claim for $s_-$ is completely symmetric.
	\end{proof}
	
	\begin{remark}
		An alternative proof of~\eqref{eq_sKmu1} may be obtained by observing that for $\P$-a.e.\ $\omega\in\Omega$, one has
		\begin{equation}\label{eq_koclGine1}
			\lim_m \quad \inf \left\{\left| \pi \left(\sigma\left(\frac{1}{m} S(X_m(\omega)) \right)\right)- \sigma(K_\mu)\right| : \pi\in S_\infty\right\} = 0
		\end{equation}	
		using theorem~3.1 from~\cite{KoltchinskiiGine-RandomIntOp2000}. Here $\sigma(K_\mu)$ stands for the spectrum of $K_\mu$, $\sigma(\frac{1}{m} S(X_m))$ stands for the spectrum of the random matrix
		$\frac{1}{m} S(X_m)$, both being considered as elements of the space $\ell^2$ of square summable sequences; and $S_\infty$ stands for the set of all possibly infinite permutations of a countable set. 
		This provides
		\[
		\lim_m s_\pm \left(\frac{1}{m} S(X_m(\omega)) \right) \geq s_\pm (K_\mu)
		\]
		for $\P$-a.e.\ $\omega\in\Omega$. Hence~\eqref{eq_sKmu1} holds because
		\[
		s_\pm \left( S(X_m(\omega))\right)= s_\pm \left(\frac{1}{m} S(X_m(\omega))\right)
		\]
		for all $\omega\in \Omega$ and 
		\[
		\lim_m s_\pm \left( S(X_m(\omega))\right)= s_\pm(X,d)
		\]
		for $\P$-a.e.\ $\omega\in\Omega$ by 
		Proposition~\ref{prop:sign convergence3}. 
		The nice feature of this argument is that it requires neither the completion of the space nor the use of continuity of the operator kernel. 
	\end{remark}

	The following result on the relationship between the signatures of the operators $K_\mu$ and $T_\mu$ is also worth being mentioned.
	
	\begin{proposition}\label{prop: countable space sign}
		Let $(X, d, \mu)$ be a metric measure space with a
		Borel probability measure $\mu$ having finite $4$-th moment and full support in $X$, and 
		the operators $K=K_\mu$ and $T=T_\mu$ are defined by~\eqref{eqn:MDS opK} 
		and~\eqref{eqn:MDS opT}, respectively. 
		Then 
		\begin{equation}\label{eq_spmKT1}
			s_\pm(K) - 1 \le s_\pm(T) \le s_\pm(K).
		\end{equation}	
		Moreover, if $\fone \not\perp \ker K$, then $s_\pm(T) = s_\pm(K)$.
	\end{proposition}
	
	\begin{proof}
		The estimates~\eqref{eq_spmKT1}
		are just Lemma~\ref{lm spm_control1}(i) and~(ii) applied with $A\eqset K$, $Q\eqset P$, and $B\eqset T$, $H_1 = H_2 \eqset L^2(X,\mu)$, $k=1$.
		Lemma~\ref{lm spm_control1}(iii) with the same notations implies also the last claim
		since $(\Img P)^\perp = \fone^\perp$ and $\fone \notin \ker K$ (otherwise 
		$\mu$ is a Dirac delta measure, $T = K = 0$, and $\ker K = \{0\}$, a contradiction), and hence
		$\fone \not\perp \ker K$ implies $(\Img P)^\perp \cap (\ker K)^\perp=\{0\}$.
	\end{proof}
	
	Note that if $\fone$ is an eigenfunction of $K$ with a nonzero eigenvalue, then $\fone \perp \ker K$.
	
	Finally, we provide the following characterization of $s_\pm (T_\mu)$, similar to Proposition~\ref{prop:sign convergence2_sep} for $s_\pm (K_\mu)$.
	
	\begin{proposition}\label{prop_sigTmu1}
		Let $(X,d)$ be a separable metric space.
		If $\mu$ has finite $4$-th moment and full support in $X$, then one has
		\begin{equation}\label{eq_sigTmu11}
			s_\pm (T_\mu) = \sup\{ s_\pm (\Pi_N S(X_N) \Pi_N ) : X_N\subset X\, \mbox{finite}, \# X_N =N \},
		\end{equation}
		where $T_\mu$ is defined by~\eqref{eqn:MDS opT}.
		In particular, $s_\pm (T_\mu)$ are the same for all $\mu$ with finite $4$-th moment and full support in $X$.
	\end{proposition}
	
	\begin{proof}
		It suffices to replicate word-to-word the proof of Proposition~\ref{prop:sign convergence2_sep} replacing the use of 
		the estimate~\eqref{eq_spmA1} by~\eqref{eq_spmP1} and~\eqref{eq_spmA2} by~\eqref{eq_spmP2}.
	\end{proof}
	
	%
	
	\section{The role of $s_+$ and $s_-$}
	
	\subsection{Embedding in a Hilbert space}
	We feel almost obliged to mention first the following theorem on existence of an isometric embedding into a Hilbert space, which is essentially just a reformulation of the classical Schoenberg theorem~\cite[theorem 3.1]{alfakih2018euclidean}. 
	
	\begin{proposition}\label{prop_Hllbert1}
		The following statements for a separable metric space $(X,d)$ are equivalent.
		\begin{itemize}
			\item[(i)] The space $(X,d)$ is isometrically embeddable in a Hilbert space.
			\item[(ii)] For some Borel probability measure $\mu$ on $X$ with full support and finite $4$-th moment, one has $s_-(T_\mu)=0$, i.e.,\ $T_\mu$ is positive semidefinite.
			\item[(iii)] For every Borel probability measure $\mu$ on $X$, 
			one has $s_-(T_\mu)=0$. 
		\end{itemize}
		In case~(ii), for an isometric embedding of $X$ in $\ell^2$, we may take the map 
		$f\colon X\to \ell^2$ 
		defined by
		\begin{equation}\label{eq_Hilbert_emb1}
			f(x)\eqset \left(\sqrt{\lambda_i} u_i(x)\right)_i, 
		\end{equation}	
		where
		$\lambda_i\in \R$ be strictly positive eigenvalues (counting multiplicities) and $u_i\in L^2(X,\mu)$ the respective eigenfunctions
		of the operator $T_\mu \colon L^2(X,\mu)\to L^2(X,\mu)$ (defined by~\eqref{eqn:MDS opT}) normalized so that $\|u_i\|_2=1$ for all $i$.
		
		In all these cases, one necessarily has $s_-(X,d)=1$.
	\end{proposition}
	
	\begin{proof}
		If~(i) holds, we identify $X$ with the image of the respective isometric embedding and $\mu$ with its push-forward through this embedding. We may assume thus without loss of generality that $X\subset H$, for some Hilbert space $H$, and $\mu$ be a measure in $H$. Then every finite set $X_N\subset X$ is isometrically embedded in a finite-dimensional Euclidean space
		(the subspace $\Span X_N\subset H$). By Theorem~\ref{th_pseudoeukl1}\footnote{We may hare referred here to the celebrated Schoenberg's theorem. We use a more general Theorem~\ref{th_pseudoeukl1} instead to be self-consistent}, for the operator
		$T_N\eqset T_{\mu_N}$ with $\mu_N$ an arbitrary Borel probability measure supported on $X_N$, we have $s_-(T_N)=0$. Thus, by Lemma~\ref{lm_approxLinOp1}(ii) with $A\eqset K_\mu$ and $H$ instead of $X$, we have
		\[
		s_-(T_\mu)\leq \liminf s_-(T_N) =0,
		\]
		thus proving~(i)$\Rightarrow$(iii).
		
		The implication~(iii)$\Rightarrow$(ii) is trivial. We prove~(ii)$\Rightarrow$(i). To this aim, let $\mu$ be as in~(ii) and let $T\eqset T_\mu$. Observe that
		by Lemma~\ref{lm_lamkT2a} (with $A\eqset T$, $a\eqset k_T$),
		one has
		\begin{equation}\label{eq_lamkT21}
			\sum_{i}\lambda_i u_i(x) u_i(y) = k_T(x,y),
		\end{equation}
		with the equality formally valid only in the sense of $L^2(X\times X,\mu\otimes\mu)$. 
		But since $T$ is positive definite, by Mercer's theorem, the convergence in~\eqref{eq_lamkT21} is uniform so that
		this equation holds for all $(x,y)\in X\times X$.
		We act now exactly as in the proof of the first part of Theorem~\ref{th_pseudoeukl1}.
		Namely, from~\eqref{eq_lamkT21}, we get
		\begin{equation}\label{eq_lamkT11}
			\sum_i \lambda_i u_i^2(x) = k_T(x,x)
		\end{equation}
		for all $x\in X$. Now,~\eqref{eq_lamkT11} and~\eqref{eq_lamkT21} imply
		\begin{align*}
			\sum_i \lambda_i \left(u_i(x) - u_i(y)\right)^2 
			&= \sum_i \lambda_i u_i^2(x) + \sum_i \lambda_i u_i^2(y) - 2 \sum_i \lambda_i u_i(x) u_i(y) \\
			&= k_T(x, x) + k_T(y, y) - 2 k_T(x, y),
		\end{align*}
		which in view of~\eqref{eq_defT2}
		gives
		\begin{equation}\label{eq_lamkT31}
			\begin{aligned}
				\sum_i \lambda_i \left(u_i(x) - u_i(y)\right)^2 =d^2(x, y)
			\end{aligned}
		\end{equation}
		for all $(x,y)\in X\times X$. 
		The latter implies for the map $f$ defined by~\eqref{eq_Hilbert_emb1} the relationship holds 
		\begin{align*}
			\|f(x)-f(y)\|_{\ell_2}^2 & =
			\sum_i \lambda_i \left(u_i(x) - u_i(y)\right)^2 =d^2(x, y),
		\end{align*}
		i.e.,\ $f$ is an isometry onto the image, proving~(i).
		
		The final claim of the statement follows from the estimate $s_-(X,d)\geq 1$ valid for every metric space, and 
		from Theorem~\ref{th_spm1}(iv) which implies 
		the estimates
		\[
		s_-(X,d) = s_-(K_\mu)\leq s_-(T_\mu)+1=1
		\]
		for every $\mu$ Borel probability measure with full support and finite $4$-th moment in $X$.
		This completes the proof.
	\end{proof}
	
	\begin{remark}
		When the metric spaces $(X,d)$ is isometrically embeddable in a Hilbert space, then $s_-(X,d)=1$, but the converse is false. 
		There exist even finite metric spaces $(X,d)$ not embeddable isometrically in a Hilbert space, which still satisfy $s_-(X,d)=1$, as the following example shows. 
	\end{remark}
	
	\begin{example}\label{ex_sminfin1}
		Let $X\eqset \{1,2,3,4$\} with $d(i,j)\eqset 2$ unless either $i=j$, in which case $d(i,i)=0$ or either $i=1,2,3$ and $j=4$, or, symmetrically,
		$j=1,2,3$ and $i=4$. In the latter cases $d(1,4)=d(2,4)=d(3,4)=1$.
		Note that the points $1,2,3$ embed isometrically in the Euclidean plane $\R^2$ as vertices of some equilateral triangle of sidelength $a=2$. However,
		the points $1,2,3,4$ cannot embed isometrically in $\R^3$ (hence in any Hilbert space) 
		since $4$ should be otherwise the midpoint of each of the segments $[i j]$, $i, j=1,\ldots, 3$.
		Nevertheless, $s_-(X,d)=1$ and $s_+(X,d)=3$. 
	\end{example}
	
	The following easy example shows the metric space isometrically embeddable in an infinite-dimensional Hilbert space. 
	
	\begin{example}
		Let $X\eqset \N$, with $d(i,j)\eqset 1$ when $i\neq j$ (and, of course, $d(i,i)=0$). Then $S(X_N)$ for every $N$-element subset $X_N\subset X$
		has an eigenvalue $1/2$ with multiplicity $N-1$ and one eigenvalue $-(N-1)/2$. Then $s_+(S(X_N))=N-1$ and $s_-(S(X_N))=1$ implies
		$s_+(X)=+\infty$ and $s_-(X)=1$. Note that this space can be easily isometrically embedded in a Hilbert space (e.g., in the space $\ell^2$ of square summable sequences, via the map $k\mapsto e_k/\sqrt{2}$, $e_k$ standing for the $k$-th coordinate vector).
	\end{example}
	
	\begin{example}\label{ex_spmEucl1}
		Let $\R^n$ stand for the usual Euclidean $n$-dimensional space.
		Then $s_-(\R^n)=1$, and $s_+(\R^n)=n$, and hence for a metric space $(X,d)$ to be isometrically embeddable in $\R^n$, 
		it is necessary and sufficient that $s_-(T) = 0$ and $s_+(X,d) \leq n$  (clearly, in this case
		$s_+(X,d) =n$ unless $(X,d)$ can be isometrically embedded in some $\R^m$ with $m<n$).
	\end{example}
	
	
	\subsection{Embedding in pseudo-Euclidean spaces}
	
	The pseudo-Euclidean space $\R^{n,p}$ with signature $(n,p)$ is the linear space of
	vectors $\R^{n+p}$ equipped with the bilinear form 
	\[
	(u,v)\eqset -\sum_{i=1}^n u_iv_i + \sum_{j=n+1}^{n+p} u_jv_j.
	\]
	The latter defines a pseudo-distance 
	$d_{n,p}(u,v) \eqset \sqrt{(u-v,u-v)}$
	on every set $\Sigma\subset \R^{n,p}$ such that the Minkowski difference
	$\Sigma - \Sigma \eqset \{u-v : u\in\Sigma, v\in \Sigma\}$ belongs to the positive cone
	\[
	C_{n,p}\eqset \left\{v\in \R^{n,p} : (v,v) \eqset -\sum_{i=1}^n v_i^2 + \sum_{j=n+1}^{n+p} v_j^2 \ge 0\right\}.
	\]
	The following statement is valid.
	
	\begin{theorem}\label{th_pseudoeukl1}
		Let $\mu$ be a Borel probability measure with full support and finite $4$-th moment in the metric space $(X,d)$. 
		Suppose that 
		$s_-(T_\mu)=n\in \N$ and $s_+(T_\mu)=p\in \N$. 
		Let $\lambda_i\in \R$ be nonzero eigenvalues (counting multiplicities) and $u_i\in L^2(X,\mu)$ be the respective eigenfunctions
		of the operator $T_\mu \colon L^2(X,\mu)\to L^2(X,\mu)$ (defined by~\eqref{eqn:MDS opT}), normalized so that $\|u_i\|_2=1$, $i=1,\ldots, n+p$. Without loss of generality, assume $\lambda_i<0$ for $i=1,\ldots n$
		and $\lambda_i>0$ for $i=n+1,\ldots n+p$.
		Then the map $f\colon X\to \R^{n,p}$ 
		defined by
		\[
		f(x)\eqset \sum_{i = 1}^{n+p} \sqrt{|\lambda_i|} u_i(x), 
		\]
		is an isometry onto the image such that its image
		$\Sigma\eqset f(X)$ satisfies $\Sigma-\Sigma\subset C_{n,p}$.
		In particular, every finite metric space can be isometrically embedded into some pseudo-Euclidean space with the image $\Sigma$ such that $\Sigma-\Sigma$ belongs to the positive cone.
		
		Vice versa, if $\mu$ is a Borel probability measure with finite $4$-th moment (not necessarily of full support in $X$) and there is an $f\colon X\to \R^{n,p}$ isometry onto the image such that
		$\Sigma\eqset f(X)$ satisfies $\Sigma-\Sigma\subset C_{n,p}$, then $s_-(T_\mu)\leq n$ and $s_+(T_\mu)\leq p$.
	\end{theorem}
	
	\begin{proof}
		By Lemma~\ref{lm_lamkT2a} (with $A\eqset T$, $a\eqset k_T$), one has
		\begin{equation}\label{eq_lamkT2}
			\sum_{i = 1}^\nu \lambda_i u_i(x) u_i(y) = k_T(x,y),
		\end{equation}
		where $\nu\eqset n+p$, with the equality in~\eqref{eq_lamkT2} formally valid only in the sense of $L^2(X\times X,\mu\otimes\mu)$. 
		Since the functions on both sides of this equality are continuous, then~\eqref{eq_lamkT2} is also valid pointwise, i.e.,\ for all $(x,y) \in X\times X$.
		In particular, we also get
		\begin{equation}\label{eq_lamkT1}
			\sum_{i = 1}^\nu \lambda_i u_i^2(x) = k_T(x,x)
		\end{equation}
		for all $x\in X$. 
		
			
		From~\eqref{eq_lamkT1} and~\eqref{eq_lamkT2}, we get
		\begin{align*}
			\sum_{i = 1}^\nu \lambda_i \left(u_i(x) - u_i(y)\right)^2 
			&= \sum_{i = 1}^\nu \lambda_i u_i^2(x) + \sum_{i = 1}^\nu \lambda_i u_i^2(y) - 2 \sum_{i = 1}^\nu \lambda_i u_i(x) u_i(y) \\
			&= k_T(x, x) + k_T(y, y) - 2 k_T(x, y)
		\end{align*}
		for all $(x,y)$. The identity~\eqref{eq_defT2}
		gives then
		\begin{equation}\label{eq_lamkT3}
			\begin{aligned}
				\sum_{i = 1}^\nu \lambda_i \left(u_i(x) - u_i(y)\right)^2 =d^2(x, y)
			\end{aligned}
		\end{equation}
		for all $(x,y)$. 
		We observe now that one can view $f$ as a map between $X$ and $\R^{n,p}$ with~\eqref{eq_lamkT3} reading as 
		\begin{align*}
			d_{n,p}(f(x),f(y))^2 & =
			-\sum_{i=1}^n |\lambda_i| \left(u_i(x) - u_i(y)\right)^2 + \sum_{j=n+1}^{n+p} |\lambda_j| \left(u_j(x) - u_j(y)\right)^2 \\
			& = \sum_{i = 1}^\nu \lambda_i \left(u_i(x) - u_i(y)\right)^2 =
			d^2(x,y).
		\end{align*}
		Hence, $f$ is an isometry onto the image $\Sigma\eqset f(X)$ with $\Sigma-\Sigma\subset C_{n,p}$,
		concluding the proof of the first part.
		
		For the converse statement, we consider $X_N = \{x_i\}_{i=1}^N \subset \supp \mu$. 
		Identifying $X_N$ with $f(X_N)$, we may assume that $X_N\subset \R^{n,p}$ with
		$X_N - X_N\subset C_{n,p}$. Then $T_N \eqset \Pi_N S(X_N) \Pi_N$ is the matrix
		\[
		T_{ij} = ( \bar x_i, \bar x_j )_{n,p}, \quad\mbox{where } \bar x_i \eqset x_i- \bar{x}, \quad
		\bar{x}\eqset \frac{1}{N} \sum_{i=1}^N x_i .
		\]
		Thus, one can view $T_N$ as the difference between two matrices, $T_N = T_+ - T_-$, where
		\[
		(T_-)_{ij} = (P_{n}^-\bar x_i) \cdot (P_{n}^-\bar x_j), \qquad (T_+)_{ij} = (P_{p}^+\bar x_i) \cdot (P_{p}^+\bar x_j),
		\]
		$P_{n}^-$ (resp.\ $P_{p}^+$) standing for the projections from $\R^{n+p}$ to $\R^n$ identified with the subspace of $\R^{n+p}$ with all coordinates zero except the first $n$ ones (resp.\ to $\R^p$ identified with the subspace of $\R^{n+p}$ with all coordinates zero except the last $p$ ones). Also, $u\cdot v$ stands for the usual dot product between two vectors in $\R^{n+p}$. 
		Both $T_-$ and $T_+$ are clearly positive semidefinite with $s_+(T_-)\leq n$, $s_+(T_+)\leq p$, and $s_-(T_-)=s_-(T_+) = 0$. 
		From Lemma~\ref{lem:sign_comparison}(i), we get that
		\[
		s_+(T_N) \leq s_+(T_+) + s_+(-T_-) = s_+(T_+) + s_-(T_-) \le p.
		\]
		A completely symmetric reasoning yields $s_-(T_N)\leq n$.
		Then Proposition~\ref{prop_sigTmu1} applied to the space $(\supp \mu, d)$ yields the desired claim
		\[
		s_+(T_\mu) \le p, \quad s_-(T_\mu) \le n ,
		\]
		since $L^2(X, \mu)$ is isomorphic to $L^2(\supp \mu, \mu)$.
	\end{proof}
		
	Note that not every subset $\Sigma$ of a pseudo-Euclidean space $\R^{n,p}$ is an image of an isometric embedding of a metric space.
	In particular, the condition $\Sigma-\Sigma\in C_{n,p}$ is a strong restriction, see the example below.
	
	\begin{example}
		If $n=1$, and $\Sigma\subset \R^{n,p}$ satisfies 
		$\Sigma-\Sigma\subset C_{n,p}$, then the intersection of $\Sigma$ with every line parallel to $x_1$ axis must be either empty or a singleton. 
		In other words, $\Sigma$ belongs to a hypersurface which, if seen in $\R^{n+p}=\R^{p+1}$, is a graph $G$ of some function: $x_1 = g(x_2, \ldots, x_{p+1})$. 
		Clearly, $g\colon \R^p\to \R$ is a $1$-Lipschitz function, since $G - x \subset C_{n,p}$ for every $x\in G$.
	\end{example}

	\subsection{On feasible signatures}
	
	The following result shows that for a generic finite metric space there exists an arbitrarily small perturbation of the distance which preserves $s_+$ and makes $s_-$ maximum possible.
	
	\begin{proposition}\label{prop_negsign1}
		Let $(X, d)$ be a finite metric space with $X = \{x_1, \dots, x_N\}$ such that the triangle inequality is strict for any distinct triplet from $X$. Consider the matrix $T\eqset \Pi_N S(X) \Pi_N$. 
		Then for any sufficiently small $\varepsilon > 0$, there is a distance $d_\varepsilon$ over $X$ such that 
		\[
		\max_{i, j} |d(x_i, x_j) - d_\varepsilon(x_i, x_j)| \le \varepsilon
		\]
		and the corresponding matrix $T_\varepsilon$ has signature 
\[		
s_+(T_\varepsilon) = s_+(T), \quad s_-(T_\varepsilon) = N - 1 - s_+(T).
\]	
\end{proposition}
	
	\begin{proof}
		Fix linearly independent vectors $v_1, \dots, v_N \in \R^N$.
		For any $\varepsilon > 0$ define 
		\[
		g_\varepsilon(x_i, x_j) = d^2(x_i, x_j) - \varepsilon \|v_i - v_j\|^2.
		\]
		Note $g_\varepsilon \to d^2$ uniformly as $\varepsilon \to 0^+$. For sufficiently $\varepsilon>0$, we have $d_\varepsilon \eqset \sqrt{g_\varepsilon}\geq 0$.
		Moreover, since for any distinct $i, j, k$, the strict triangle inequality 
		\[
		d(x_i, x_j) < d(x_i, x_k) + d(x_k, x_j)
		\]
		holds, then for all sufficiently small $\varepsilon$, one has the triangle inequality
		\[
		d_\varepsilon(x_i, x_j) \le d_\varepsilon(x_i, x_k) + d_\varepsilon(x_k, x_j).
		\]
		Thus, $d_\varepsilon$ is a distance over $X$ for small $\varepsilon>0$. We set then $X_\varepsilon \eqset (X,d_\varepsilon))$.
		
		Now, it is easy to see that the matrix corresponding to $(X,d_{\varepsilon})$ has the form
		\[
		T_\varepsilon\eqset \Pi_N S(X_\varepsilon)\Pi_N = T - \varepsilon G ,
		\]
		where $G \eqset \begin{pmatrix} \left(v_i - \bar{v}, v_j - \bar{v} \right) \end{pmatrix}_{i,j=1}^N$ with $\bar{v} \eqset \frac{1}{N} \sum_{i=1}^N v_i$.
		In particular, $G$ is a Gram matrix with $s_+(G) = N - 1$, $s_-(G) = 0$, thus due to Lemma~\ref{lem:sign_comparison}(i) one has
		\begin{align*}
			s_+(T_\varepsilon) &\le s_+(T) + s_+(-\varepsilon G) = s_+(T), \\
			s_-(T_\varepsilon) &\ge s_-(- \varepsilon G) - s_-(-T) = N - 1 - s_+(T).
		\end{align*}
		Moreover, for small enough $\varepsilon$, one has $s_+(T_\varepsilon) = s_+(T)$ by Lemma~\ref{lem:sign_comparison}(ii). 
		Since $\ker T$ is non-trivial, we have $s_-(T_\varepsilon) \le N - 1 - s_+(T_\varepsilon) = N - 1 - s_+(T)$. Thus, we get the equality for $s_-(T_\varepsilon)$.
		The claim follows.
	\end{proof}
	
	\begin{corollary}\label{co_sign_np1}
		For any $p\geq 2$ and $n \ge 1$, there is an $(n + p + 1)$-point metric space $(X,d)$ such that the corresponding matrix $T$ has signature $s_+(T) = p$, $s_-(T) = n$.
	\end{corollary}
	
	\begin{proof}
		It is enough to apply the above Proposition~\ref{prop_negsign1} to $N\eqset n + p + 1$ distinct points in a general position (i.e.\ spanning $\R^p$) on the unit sphere in $\R^p$, endowed with the Euclidean distance.
	\end{proof}
	
	\begin{remark}
		Under the conditions of the above Corollary~\ref{co_sign_np1}, one clearly has
		$s_+(X,d)\in \{p, p+1\}$, $s_-(X,d)\in \{n, n+1\}$ in view of Theorem~\ref{th_spm1}(iv). Thus, in particular, this corollary gives a construction of hollow symmetric nonnegative (HSN) matrices (in this particular case even stronger, squared distance matrices)
		with arbitrarily large number of positive eigenvalues. 
		This provides an alternative proof to the main theorem of~\cite{Charles2013} (the latter seems to be much more complicated and gives just HSN matrices, not necessarily related to any distances).
	\end{remark}

	We are also able to provide the following example of an infinite metric space $(X,d)$ with $s_-(X,d)=+\infty$ and $s_+(X,d)$ finite (arbitrary up to the error of $1$).

	\begin{example}
		Let $p\in \N$ be fixed. In the following construction we assume the pseudo-Euclidean spaces $\R^{n,p}$ with different $n\in \N$ to be
		hierarchically embedded one into another, in such a way that a vector $x\in \R^{n,p}$ is identified
		with the vector $y\in \R^{m,p}$, $m>n$, defined by
		\begin{align*}
			y_i:=x_i, & \quad i=1,\ldots, n,\\
			y_i:=0, & \quad i=n+1,\ldots, m,\\
			y_{m+j}:= x_{n+j}, & \quad j=1,\ldots, p.
		\end{align*}
		This gives a natural embedding of $\R^{n,p}$ into $\R^{m,p}$. 
		In the same way, we identify $\R^n$ (resp.\ $\R^p$) with the subspace of $\R^{n,p}$ with last $p$ coordinates (resp.\ first $n$ coordinates) zero. 
		As in the proof of Theorem~\ref{th_pseudoeukl1}, we denote
		by $P_{n}^-$ (resp.\ $P_{p}^+$) the projections from $\R^{n+p}$ to $\R^n$ (resp.\ $\R^p$), both identified with  subspaces of $\R^{n+p}$,
		defined by $(P_{n}^- z)_i\eqset z_i$, and $(P_{n}^- z)_j\eqset 0$ (resp.\ $(P_{p}^+ z)_i\eqset 0$, $(P_{p}^+ z)_j\eqset z_j$), where $i=1,\ldots, n$, $j=n+1,\ldots, n+p$).
		
		We are going to construct inductively a sequence of increasing sets 
		\[X_N = \{z^1,\ldots, z^N\} \subset \R^{N-1,p},\quad  N \ge p + 1,\] 
		with special properties.
		For brevity, we denote 
		\begin{align*}
			S_N \eqset S(X_N), \quad 
			T_N \eqset \Pi_N S_N \Pi_N .
		\end{align*}
		As in the proof of Theorem~\ref{th_pseudoeukl1}, one has
		\begin{align*}
			T_N & = T_N^+ - T_N^-, \quad\mbox{where}\\
			(T_N^-)_{ij} &\eqset (P_{N-1}^- \bar z^i)\cdot (P_{N-1}^- \bar z^j), \\
			(T_N^+)_{ij} &\eqset (P_p^+ \bar z^i)\cdot (P_p^+ \bar z^j), \quad i,j=1, \ldots N,\\
			\bar z^i &\eqset z^i - \frac{1}{N} \sum_{j=1}^{N} z^j,
		\end{align*}
		where the dot product is the Euclidean one. 
		Clearly, $T_N^\pm$ are both positive semidefinite.
		We will construct $(X_N)_{N=p+1}^\infty$ in such a way that
		\begin{itemize}
			\item[(i)] both $P_{N-1}^- X_N$ and $P_p^+ X_N$ are sets of points in a general position in $\R^{N-1}$ and $\R^p$ respectively (i.e.\ not contained in any hyperplane in these spaces),	
			\item[(ii)] $d_{N-1,p}$ is a distance over $X_N$ and for every triple of distinct points in $X_N$, the triangle inequality for $d_{N-1,p}$ is strict,
			\item[(iii)] $P_p^+ X_N \subset S$, where $S$ is the unit sphere in $\R^p$,
			\item[(iv)] $u \cdot P_p^+ (z^k - z^1) < d_{N-1,p}(z^k, z^1) |u|$ for all $k = 2, \dots, N$ and non-zero $u \in T_{y^1}(S)$, where $T_{y^1}(S)$ is the tangent space to $S$ at point $y^1:=P_p^+ z^1$, and $|\cdot|$ stands for the Euclidean norm in $\R^p$.
		\end{itemize}
		In view of~(i) we will have then
		\[
		s_+(T_N^-) = N - 1, \quad s_+(T_N^+) = p, \quad s_-(T_N^-) = s_-(T_N^+) = 0.
		\]
		Thus from $T_N = T_N^+ - T_N^-$ with the help of Lemma~\ref{lem:sign_comparison}(i) we get
		\begin{equation}\label{eq_Tpn1}
		\begin{aligned}
			s_-(T_N) &\ge s_-(- T_N^-) - s_- (- T_N^+)\\
			&= s_+(T_N^-) - s_+ (T_N^+) = N - p - 1,
		\end{aligned}
		\end{equation}
		as well as $s_+(T_N) \le s_+(T_N^+) + s_-(T_N^-) = p$.

		We start with $N \eqset p + 1$: take $p+1$ points $\{y^1,\ldots, y^{p+1}\}$ in a general position on the unit sphere $S \subset \R^p$.
		Now, let $\{x^1,\ldots, x^{p+1}\} \subset \R^p$ be in a general position and with the norms $|x^k|$, $k=1,\ldots, p+1$, so small that for $d_{p,p}$ satisfies strict triangle inequalities over the set
		\[
		X_{p+1} \eqset \{z^1,\ldots, z^{p+1}\}, \quad z^k \eqset (x^k, y^k) \in \R^{p,p} ,
		\]
		the property~(iv) holds for $N:=p+1$ 
		and $s_+(T_{p+1}) = s_+(T_{p+1}^+) = p$ (where the last equality is due to Example~\ref{ex_spmEucl1}).
		Then, in particular,
		\begin{equation}\label{eq:Sp+1}
			s_+ (S(X_{p+1})) = p 
		\end{equation}
		in view of 
		Theorem~\ref{th_spm1}(iv) and the fact that $s_- (S(X_{p+1})) \ge 1$. 
		The possibility to satisfy~(iv) follows from the strict convexity of the unit ball in $\R^p$ which implies
		\[
		u \cdot (y^k - y^1) < |y^k - y^1| \cdot |u|
		\]
		for all $k\neq 1$ and nonzero $u\in T_{y^1}(S)$,  and 
		\[
		d_{p,p}(z^k, z^1) \to |y^k - y^1| \quad \mbox{ as $x_k \to 0$, $x_1\to 0$}.
		\]

		Once $X_N = \{z^1,\ldots, z^N\} \subset \R^{N-1,p}$ is constructed, we denote 
		\[
		y^k\eqset P_p^+ z^k\in S, \quad 
		x^k\eqset P_{N-1}^- z^k.
		\]	
		We will choose $z^{N+1} = (x^{N+1}, y^{N+1})$ such that $y^{N+1} \in S$ and $x^{N+1} \in \R^N$ has the form
		\begin{align*}
			(x^{N+1})_i &\eqset (x^1)_i, \quad i=1, \ldots, N-1,\\
			(x^{N+1})_N &\eqset \varepsilon_N ,
		\end{align*}
		where $\varepsilon_N > 0$.
		First, note that 
		\[
		\mathrm{dist}\,\left(\frac{y - y^1}{|y - y^1|}, T_{y^1}(S)\right) \to 0 \quad \mbox{as $y \to y^1$},
		\] 
			where $\mathrm{dist}$ stands for the distance in $\R^p$ between a point and a set. 
		Then, due to~(iv), one can choose $y^{N+1}$ close enough to $y^1$, so that for an auxiliary point $\tilde{z}^{N+1} \eqset (x^1, y^{N+1})$ one has
		\begin{align*}
	(\tilde z^{N+1} - z^1, z^1 - z^k)_{N,p} &= (y^{N+1} - y^1) \cdot (y^1 - y^k) \\
	& < d_{N,p}(z^k, z^1) |y^{N+1} - y^1|, \\
	(z^1 - \tilde{z}^{N+1}, \tilde{z}^{N+1} - z^k)_{N,p} &= (y^1 - y^{N+1}) \cdot (y^{N+1} - y^k) \\
	& < d_{N,p}(z^k, \tilde{z}^{N+1}) |y^{N+1} - y^1|,
\end{align*}
and
		\begin{align*}
			|y^{N+1} - y^1|  &  < d_{N,p}(z^1, z^k)  , 
		\end{align*}
		hence
		\begin{align*}
			d_{N,p}(\tilde{z}^{N+1}, z^1) &= |y^{N+1} - y^1| < d_{N,p}(z^1, z^k) + d_{N,p}(z^k, \tilde{z}^{N+1}) , 
		\end{align*}
		for all $k = 2, \dots, N$, with $a\cdot b$ standing for the Euclidean dot product in $\R^p$.
		Thus, by Lemma~\ref{lem:strict_triangle},
		\begin{align*}
			&d_{N,p}(\tilde{z}^{N+1}, z^k) < d_{N,p}(\tilde{z}^{N+1}, z^1) + d_{N,p}(z^1, z^k), \\
			&d_{N,p}(\tilde{z}^1, z^k) < d_{N,p}(z^1, \tilde{z}^{N+1}) + d_{N,p}(\tilde{z}^{N+1}, z^k).
		\end{align*}
		for all $k = 2, \dots, N$, and hence $d_{N,p}$ satisfies the strict triangle inequality on the set $X_N \cup \{\tilde{z}^{N+1}\}$.
		Now, we can choose $\varepsilon_N$ small enough, such that $d_{N,p}$ is a distance over $X_{N+1} \eqset X_N \cup \{z^{N+1}\}$ satisfying the strict triangle inequality, and 
		\[
		u \cdot (y^{N+1} - y^1) < d_{N,p}(z^{N+1}, z^1) |u|
		\]
		for all nonzero $u \in T_{y_1} S$. 
		The latter follows from the fact that
		\[
		u \cdot (y^{N+1} - y^1) < |y^{N+1} - y^1| \cdot |u|
		\]
		due to the strict convexity of the unit ball in $\R^p$, and 
		\[
		d_{N,p}(z^{N+1}, z^1) = \sqrt{|y^{N+1} - y^1|^2 - \varepsilon_N^2} \to |y^{N+1} - y^1| \quad \mbox{ as $\varepsilon_N \to 0$}.
		\]
		Therefore, $X_{N+1}$ satisfies~(i)-(iv).
		
		To conclude, let $X \eqset \bigcup_{N = p+1}^\infty X_N$, and equip this set with the distance coinciding with $d_{N,p}$ over each $X_N$.
		Then 		\[
		s_\pm(X, d) = \lim_n s_\pm(S_N) .
		\]
		by Theorem~\ref{th_spm1}(i).
		Clearly, in view of Theorem~\ref{th_spm1}(iv) and~\eqref{eq_Tpn1}, one has 
		\[
		s_-(S_N)\geq N - p - 1,
		\]
		and hence $s_-(X,d) = +\infty$.
		On the other hand, by Theorem~\ref{th_pseudoeukl1} combined with Theorem~\ref{th_spm1}(iv), we get $s_+(X,d) \le p+1$.
		Finally, $s_+(X,d) \geq s_+(S(X_{p+1})) = p$ by~\eqref{eq:Sp+1}, thus
		\[
		s_+(X,d)\in \{p,p+1\} \quad \mbox{and} \quad s_-(X,d) = +\infty,
		\]
		concluding the example.
	\end{example}

	The following example shows another construction of an infinite metric space $(X,d)$ with $s_+(X,d)=+\infty$ and $s_-(X,d)$ finite but different from $1$. 
	
	\begin{example}\label{ex_sminfin2}
		Let $X\eqset \N$, and define the distance
		\begin{align*}
			d(i,j)\eqset \begin{cases}
				1,\ i<4,\ j=4\quad \mbox{or }i=4,\ j<4,\\
				0,\ i=j,\\
				2,\ \mbox{otherwise}.
			\end{cases}
		\end{align*}
		Note that the points $1,2,3,4$ form the same tripod from Example~~\ref{ex_sminfin1}, which cannot be embedded in any Hilbert space. 
		The subset $X_N\eqset \{1,\ldots, N\}$ has the $N\times N$ squared distance matrix
		\[
		B_N\eqset \left(
		\begin{tabular}{lllllll}
			0& 4& 4 &1 &4 &\ldots &4\\
			4& 0& 4 &1 &4 &\ldots &4\\
			4& 4& 0 &1 &4 &\ldots &4\\
			1& 1& 1 &0 &4 &\ldots &4\\
			4& 4& 4 &4 &0 &\ldots &4\\
			\ldots &\ldots &\ldots &\ldots &\ldots &\ldots &\ldots\\
			4& 4& 4 &4 &4 &\ldots &0\\
		\end{tabular}
		\right)
		\]
		and by the Haynsworth inertia additivity formula \cite[theorem~1]{haynsworth1968determination} its signature satisfies
		\[
		(s_-,s_0,s_+)(B_N) = (s_-,s_0,s_+)(B_4) + (s_-,s_0,s_+)(B_N/B_4),
		\]
		where $B_N/B_4$ stands for the Schur complement of $B_4$ in $B_N$. The explicit calculation shows
		$(s_-,s_0,s_+)(B_4)=(3,0,1)$ and
		\[
		B_N/B_4\eqset 
		\frac{4}{3}\left(
		\begin{tabular}{lllllll}
			 8& 11& 11&  11&\ldots &11\\
			11&  8& 11&  11&\ldots &11\\
			11& 11&  8&  11&\ldots &11\\
			\ldots &\ldots &\ldots &\ldots &\ldots &\ldots\\
			11& 11& 11& 11 &\ldots &8\\
		\end{tabular}
		\right).
		\]
		For $N\eqset 4+k$ the latter $k\times k$ matrix has one eigenvalue $-4$ of multiplicity $k-1$ and one strictly positive simple eigenvalue (the Perron--Frobenius one), that is, $(s_-,s_0,s_+)(B_{4+k}/B_4)=(k-1,0,1)$. 
		Thus, $(s_-,s_0,s_+)(B_{4+k})=(k+2,0,2)$ for $k \ge 1$, and therefore, $s_-(X,d)=2$ and $s_+(X,d)=+\infty$. 
	\end{example}
	
	Now, we extend the previous Example~\ref{ex_sminfin2} to show that $s_-$ can be arbitrarily large.

	\begin{example}\label{ex_sminfin3}
		Let $(X_i, d_i)$, $i = 1, \dots, m$, be a finite metric space with the squared distance matrix $D_i \eqset (d_i^2(x_j, x_k))_{j,k}$.
		Assume $\diam(X_i) \le 2 h$ for some $h > 0$.
		
		Denote for the sake of brevity $N_i \eqset \# X_i$, 
		$S_i \eqset S(X_i)$ and 
		\[
		R_i \eqset \frac{h^2}{2} \vone_{N_i} \vone^T_{N_i} + S_i = \frac{1}{2} (h^2 \vone_{N_i} \vone^T_{N_i} - D_i).
		\]
		Now we define a metric space $(X, d)$ with $X \eqset \bigsqcup_{i=1}^m X_i$ and
		\[
		d(x, y) \eqset \begin{cases}
			d_i(x, y), & x, y \in X_i, \\
			h, & x\neq y\ \text{otherwise.}
		\end{cases}
		\]
		The respective squared distance matrix is given by
		\[
		D \eqset \begin{pmatrix}
			D_1 & h^2 & \dots & h^2 \\
			h^2 & D_2 & \dots & h^2 \\
			\vdots & \vdots & \ddots & \vdots \\
			h^2 & h^2 & \dots & D_m
		\end{pmatrix}.
		\]
		Let also
		\[
		\begin{array}{rl}
			S &\eqset S(X)=- \frac{1}{2} D,\\
			R &\eqset \frac{h^2}{2}\vone_N \vone_N^T + S = \frac{1}{2} (h^2 \vone_N \vone_N^T - D) = \diag(R_1, \dots, R_m),
		\end{array}
		\]
		where $N\eqset \sum_{i=1}^m N_i$.
		Note that $s_\pm(R) = \sum_{i=1}^m s_\pm(R_i)$.
		Further, since $\frac{h^2}{2} \vone_N \vone_N^T$ is positive semidefinite rank one matrix, 
		by Lemma~\ref{lem:sign_comparison}, we get
		\[
		s_+(S) \le s_+(R) \le s_+(S) + 1, \quad s_-(S) - 1 \le s_-(R) \le s_-(S).
		\]
		
		Let us consider a specific example, where all $X_i$ for $i = 1, \dots, m-1$, are isometric with 
		\[
		D_i = \begin{pmatrix}
			0 & 4 & 4 & 1 \\
			4 & 0 & 4 & 1 \\
			4 & 4 & 0 & 1 \\
			1 & 1 & 1 & 0
		\end{pmatrix}
		\]
		as in the previous example,
		and $X_m$ is such that $d_m(x,y) = 2$ if $x \neq y$.
		Since $\diam(X_i) = 2$ for all $X_i$, one can set $h \eqset 1$.
		Thus, for $i = 1, \dots, m-1$
		\[
		R_i = \frac{1}{2} 
		\begin{pmatrix}
			1 & -3 & -3 & 0 \\
			-3 & 1 & -3 & 0 \\
			-3 & -3 & 1 & 0 \\
			0 & 0 & 0 & 1
		\end{pmatrix}
		\]
		with $\sigma(R_i) = (-5/2, 1/2, 2, 2)$ and $s_+(R_i) = 3$, $s_-(R_i) = 1$.
		Furthermore, $X_m$ can be isometrically embedded into a Euclidean space, and 
		\[
		D_m = 4 \vone_{N_m} \vone_{N_m}^T - 4 \Id_{N_m}, \quad 
		S_m = 2 \Id_{N_m} - 2, \quad 
		R_m = 2 \Id_{N_m} - \frac{3}{2} \vone_{N_m} \vone_{N_m}^T,
		\]
		with $s_+(R_m) = N_m - 1$, $s_-(R_m) = 1$.
		Thus, 
		\[
		s_+(R) = 3 (m - 1) + N_m - 1, \quad s_-(R) = m,
		\]
		and therefore,
		\[
		s_+(X,d) = s_+(S) \in \{3 m - 5 + N_m, 3 m - 4 + N_m\}, \quad 
		s_-(X,d) = s_-(S) \in \{m, m + 1\}.
		\]
		
		Note that the constructed space $X$ can be represented as a graph endowed with its intrinsic distance.
	\end{example}

	\section{Some further properties and examples of limit distance signatures}
	
	\subsection{Further examples of infinite limit distance signature}
	
	\begin{example}\label{ex:sphere}
		Let $S^n$ be the standard unit $n$-dimensional sphere equipped with its intrinsic distance $d$. Then
		$s_\pm(S^n, d) = +\infty$ (see \cite[proof of proposition~6.1]{kroshnin2022infinite}). 
		On the other hand, $s_+(S^n, d^{1/2}) = +\infty$ and $s_-(S^n, d^{1/2}) = 1$ since $(S^n, d^{1/2})$ can be isometrically embedded in the Hilbert space $\ell^2$ by proposition~6.1 from~\cite{kroshnin2022infinite}.
	\end{example}
	
	\begin{example}\label{ex:torus}
		Let $\mathbb{T}^n$ be the standard $n$-dimensional flat torus equipped with its intrinsic distance $d$.
		Then $s_\pm(\mathbb{T}^n, d) = +\infty$ due to Example~\ref{ex:sphere} above and Proposition~\ref{prop_dsign_upper1}.
	\end{example}
			
	%
	%
	
	\begin{example}
		In view of the above Proposition~\ref{prop_dsign_upper1}, for the Urysohn universal space $\mathfrak{U}$ \cite{Vershik-randomMMS04a}, one has $s_\pm(\mathfrak{U})=+\infty$, because $\mathfrak{U}$ contains an isometric copy of every separable metric space, e.g., of the unit circle equipped with its intrinsic distance. For the same reason $s_\pm(c_0)=+\infty$, where $c_0$ stands for the Banach space of vanishing sequences, equipped with its usual supremum norm.
		Again by the same reason, we have that $s_\pm(X,d)=+\infty$ when $(X,d)$ is either the cylinder $S^1\times\R$ or the revolution torus in $\R^3$, equipped with their intrinsic distances.
	\end{example}
	
	Squared distance matrices are special cases of hollow symmetric non-negative (HSN) matrices, i.e.,\ symmetric matrices with nonnegative entries and with zeros on the diagonal.
	Using the spectral theory of HSN matrices from~\cite{Charles2013}, we obtain the following result
	for countable metric spaces.
	
	\begin{proposition}\label{prop_graphinfsig1}	
		Suppose $(X,d)$ is a bounded infinite separable metric space, and there exists 
		$\delta_0>0$ such that 
		$\delta_0\leq d(x,y)$ for all $x\neq y$. 
		Then $s_+(X,d)=+\infty$, in particular, $s_+(T_\mu)=s_+(K_\mu)=+\infty$ for every 
		Borel probability measure $\mu$ on $X$ with full support. 
	\end{proposition}
	
	\begin{proof}
		For every finite subset $X_N\subset X$,
		the squared distance matrix of $X_N$ is a HSN 
		matrix with off-diagonal entries from the interval $[\delta_0^2,\delta_1^2]$, where $\delta_1$ is the diameter of $(X,d)$. Since $\delta_0>0$ and $\delta_1<+\infty$, by theorem~3.3 from \cite{Charles2013}, there exists some $n_0\in \N$ such that any HSN matrix of order at least $n_0$ with off-diagonal entries from $[\delta_0^2,\delta_1^2]$ admits at least $m$ strictly negative eigenvalues. Therefore, for any fixed $m$, there exists some $n_0$ such that for any finite $X_N\subset X$ with $\# X_N\geq n_0$, one has that $s_+(S(X_N))\geq m$, proving the claim.
	\end{proof}
	
	The assumptions of the above corollary are satisfied, for instance, when
	the metric space $(X,d)$ is induced by a 
	connected countable graph with its intrinsic metric on vertices being bounded.

	\subsection{Limits and tangent cones}
	
	\begin{proposition}\label{prop_GH1}
		Let $(X_k, d_k)$ be a sequence of separable metric spaces with $p_k\in X_k$, and the sequence of pointed metric spaces $(X_k, d_k, p_k)$ converge as $k\to\infty$ to a pointed separable metric space $(X, d, p)$ in the pointed Gromov--Hausdorff sense.
		Assume that all $(X_k, d_k)$ and $(X, d)$ satisfy the Heine--Borel property (i.e., closed balls are compact).
		Then
		\[
		s_\pm (X,d)\leq \liminf_k s_\pm (X_k,d_k). 
		\] 
	\end{proposition}
	
	\begin{proof}
		We prove the statement for $s_+$, and the proof for $s_-$ is completely analogous.
		Fix a finite set $\Sigma_N \subset X$ and let $\bar B_r(p)\subset X$ be a closed ball centered at $p$ such that $\Sigma_N\subset \bar B_r(p)$. 
		It follows from the definition of the pointed Gromov--Hausdorff convergence \cite[definition~8.1.1]{burago2001course} that for an arbitrary $\varepsilon > 0$ there are $\bar k \in \N$ and maps $f^k_\varepsilon \colon \bar B_r(p) \subset X \to
		\bar B_{r+\varepsilon}(p_k) \subset X_k$ with $f_\varepsilon^k(p) = p_k$ such that 
		\[
		\left|d_k(f_\varepsilon^k(x), f_\varepsilon^k(y)) - d(x,y) \right| \leq \varepsilon
		\]
		for all $x,y\in \bar B_r(p)$ and $k\geq \bar k$.
		Choosing $\varepsilon>0$ sufficiently small, by Lemma~\ref{lem:sign_comparison}(ii) one has
		\[
		s_+(S(f_\varepsilon^k(\Sigma_N)))\geq s_+(S(\Sigma_N))
		\]
		for all $k\in \N$ sufficiently large. 
		We obtain the bound
		\[
		\liminf_k s_+(X_k,d_k)\geq \liminf_k s_+(S(f_\varepsilon^k(\Sigma_N)))\geq s_+(S(\Sigma_N)) .
		\]
		Taking a supremum with respect to all finite subsets $\Sigma_N \subset X$, we get the claim.
	\end{proof}
	
	\begin{corollary}
		Suppose that the metric space $(X,d)$ with Heine--Borel property has a tangent cone $(T_p, \tilde d)$ at $p\in X$.
		Then $s_\pm(X,d)\geq s_\pm (T_p, \tilde d)$. In particular, if $(X,d)$ is a smooth Riemannian manifold equipped with its intrinsic distance, then $s_+(X,d)\geq \dim X$.
		
		Similarly, if $(X,d)$ has a (Gromov--Hausdorff) asymptotic cone $(C_p, \tilde d)$ at $p\in X$, then $s_\pm(X,d)\geq s_\pm (C_p, \tilde d)$.
	\end{corollary}
	
	\begin{proof}
		Recall that the pointed metric space $(T_p, \tilde d, p)$ (resp.,\ $(C_p, \tilde d, p)$) is the pointed Gromov--Hausdorff limit of
		the sequence $(X, kd, p)$ (resp.,\ $(X, d/k, p)$) as $k\to +\infty$. Apply the above Proposition~\ref{prop_GH1}, having in mind that
		\[
		s_\pm(X,kd)=s_\pm(X,d/k)=s_\pm(X,d).
		\]
		If $(X,d)$ is a smooth Riemannian manifold,
		it suffices to recall that $T_p=\R^n$ with $n=\dim X$, and $(T_p,\tilde d)$ is isometric to $\mathbb{R}^n$ with the Euclidean distance so that $s_+(T_p,\tilde d)=s_+(\R^n)=n$
		by Example~\ref{ex_spmEucl1}.
	\end{proof}

	\section{Limit signature of the Rado graph}
	
	\subsection{Preliminaries on the spectral theory of random matrices and the Wigner semi-circle law}
	Random matrices play an important role in the study of the Rado graph, so we outline the relevant results on random matrices used in our paper.
	
	Recall the following terminology in matrix spectral distribution. Let $W_N \in \R^{N \times N}$ be a real symmetric matrix. Denote by $\lambda_1 \leq \lambda_2 \leq \dots \leq \lambda_N$ the eigenvalues of $W_N$ counting multiplicity. Then the empirical spectral distribution (ESD) of the normalized matrix $\dfrac{1}{\sqrt{N}} W_N$ is the probability measure
	\begin{align}
		\label{eqn:esd def}
		\mu_N \eqset \dfrac{1}{N}\sum_{j=1}^N \delta_{\lambda_j/\sqrt{N}}.
	\end{align}
	
	Now, we consider a sequence of random matrices $(W_N)_{N=1}^\infty$ constructed as follows. Let $(Y_k)_{k=1}^\infty$ be a sequence of i.i.d.\ random variables, and $(Z_{ij})$ with $1\leq i<j$ be a family of i.i.d.\ random variables. For each $N\geq 1$, $W_N$ is a symmetric matrix of order $N$ with entries given by
	\begin{align*}
		(W_N)_{k,k}=Y_k,\ (W_N)_{i,j}=(W_N)_{j,i}=Z_{ij},\ \quad\mbox{for all } 1\leq k\leq N,\ 1\leq i<j\leq N.
	\end{align*}
	Thus, $W_N$ can be viewed as upper left submatrices of order $N$ of the same infinite size random matrix. 
	For each $N\geq 1$, by taking the ESD of the normalized random matrix $\dfrac{1}{\sqrt{N}} W_N$, we obtain a random probability measure of the form \eqref{eqn:esd def}. The following version of the Wigner semi-circle law from \cite[theorem 2.5]{Bai2010} indicates the ESDs of these normalized random matrices converge to the semi-circle distribution.
	
	\begin{proposition}[Wigner semi-circle law]
		\label{theorem:semicircle law}
		Let $W_N$ be the order $N$ random real symmetric matrix such that the diagonal terms are i.i.d.\ random variables, and the off-diagonal terms are i.i.d.\ random variables with variance $\sigma^2 > 0$, constructed as above. Then with probability $1$, the ESDs of normalized random matrices $\dfrac{1}{\sqrt{N}} W_N$ weakly converge to the semi-circle law, i.e.,\ to a measure with density
		\[
		\rho_\sigma(x) = \begin{cases}
			\frac{1}{2 \pi \sigma^2} \sqrt{4 \sigma^2 - x^2}, & -2 \sigma \le x \le 2 \sigma,\\
			0, & \text{otherwise}.
		\end{cases}
		\]
	\end{proposition}
	
	In particular, the above theorem implies that the signatures of the matrices $(W_N)_{N=1}^\infty$ almost surely satisfy
	\begin{align}
		\lim_{N\rightarrow\infty} s_\pm(W_N) = \infty, \quad
		\lim_{N\rightarrow\infty} \dfrac{s_+(W_N)}{s_-(W_N)} = 1.
	\end{align}

	\subsection{Signatures of the MDS operators and limit signature of the Rado graph}
	\label{sec:mds op rado}
	
	The Rado graph $R$ is a homogeneous, countable, and universal graph in the sense that it contains all finite graphs as its induced subgraphs. The Rado graph also has the following finite extension property: given any two finite disjoint subsets $U,V$ of $R$, there exists an $x\in R\setminus(U\cup V)$ such that $x$ is adjacent to all points in $U$, but not adjacent to any point in $V$. The Rado graph is known to be unique up to graph isomorphism, see \cite[chapter VII]{graham2013mathematics} for details.
	Let $d$ be the intrinsic (graph) metric on $R$, and $\mu$ be any Borel probability measure on $(R,d)$ having full support.
	The above properties immediately imply that the intrinsic metric $d$ of $R$ is bounded with $d \leq 2$.
	
	We also recall that the Rado graph admits the following random construction (infinite Erd{\H o}s--R{\'e}nyi graph). 
	Fix any real number $0<p<1$. Let $(Z_{ij})$ with $i<j\in \N$ be a family of i.i.d.\ Bernoulli random variables: $Z_{ij} \sim Be(p)$.
	Then we define a random graph $G$ with the set of vertices $\N$ and the random adjacency matrix $\mathcal{A}$ defined by
	\begin{align}
		\label{eqn:infinite Bernoulli matrix}
		\mathcal{A}_{ij}=\mathcal{A}_{ji}=Z_{ij},\ \mathcal{A}_{ii}=0,\ \quad\mbox{for all } 1\leq i<j.
	\end{align}
	With probability $1$, this construction yields a graph isomorphic to the Rado graph---thus, it can be considered as a ``random enumeration'' of $R$.
	Furthermore, for each $N\geq 1$, let $S_N$ be the random matrix of order $N$ such that
	\begin{equation}\label{eq_SNRado1}
		(S_{N})_{ij} = \dfrac{3}{2}\mathcal{A}_{ij} - 2,\quad (S_N)_{ii} = 0 \quad\mbox{for all $1\leq i\neq j\leq N$}.
	\end{equation}
	Then almost surely one has $(S_{N})_{ij} = - \frac{d^2_G(i, j)}{2}$ for $1 \le i, j \le N$, where $d_G$ is the intrinsic metric on $G$. Therefore, provided that $G$ is isomorphic to $R$, there is a sequence of finite subsets $R_N \subset R$ such that $R_N \nearrow R$, $\# R_N = N$, and $S_N = S(R_N)$. 
	Since the random matrices $(S_N)_{N\geq 1}$ also satisfy the assumptions of Proposition~\ref{theorem:semicircle law}, we obtain the following theorem on the limit distance signature of $(R,d,\mu)$.
	
	\begin{theorem}
		\label{theorem:rado sign}
		One has
		\[
		s_\pm (R, d) = +\infty
		\]
		and $s_\pm (K) = s_\pm (T) = +\infty$ for every probability measure $\mu$ over $R$ having full support. 
		Moreover, the only functions in $L^2(R,\mu)$ in the kernel of $T$ are constants.
	\end{theorem}
	
	\begin{proof}
		By the Wigner semi-circle law (Proposition~\ref{theorem:semicircle law}), with probability $1$, we have
		\[
		\lim_{N\rightarrow\infty} s_\pm(S_N)=+\infty.
		\]
		On the other hand, by the random construction of the Erd{\H o}s--R{\'e}nyi--Rado graph, almost surely there is a sequence of finite subsets $R_N \subset R$ with $R_N\nearrow R$ as $N\to\infty$ such that 
		\[
		S_N = S(R_N).
		\] 
		Thus by Theorem~\ref{th_spm1}(i), we have 
		$s_\pm(R,d) = +\infty$, which by Theorem~\ref{th_spm1}(v) implies $s_\pm(K)=+\infty$.
		Finally, Proposition~\ref{prop: countable space sign} yields $s_\pm(T)=+\infty$.
		
		It remains to compute the dimension of the kernel $\ker T$ of $T$. Since $T=PKP$, clearly $\fone \in \ker T$. 
		Since $P$ is self-adjoint in $L^2(R,\mu)$, we have $0\neq f \in \fone^{\perp}\cap \ker T$ if and only if $K(f)$ is collinear to the constant function $\fone$ in $L^2(R,\mu)$. Number the vertices of $R$ by positive integers. The fact $\mu$ has full support implies $(Kf)_i=(Kf)_j$ for all $i,j\in\N$. The sequence $(Kf)_i$ has an infinite ``matrix multiplication'' representation. More specifically, set
		\[
		S = -\dfrac{1}{2}\left(d^2(i,j)\right)_{i,j=1}^{\infty},\ v_j=f_j\mu(\{j\}).
		\]
		We then have 
		\begin{align*}
			(Kf)_i = \sum_{j=1}^{\infty} S_{ij}v_j.
		\end{align*}
		For any $n\geq 1$, define
		\begin{align*}
			v_n^+=\sum_{j\leq n,\ v_j\geq 0} v_j,\ v_n^-=\sum_{j\leq n,\ v_j< 0} v_j
		\end{align*}
		Because $\mu$ has full support and $f\neq 0$, we see $v_n^+\nearrow v^+>0$ and $v_n^-\searrow v^-<0$. Since $f\in L^2(R,\mu)$ and $\mu(R)=1$, the sequence $(v_j)$ is absolutely convergent and $v^+,v^-$ are finite. Using the finite extension property of the Rado graph, for any fixed $n\geq 1$, we can find $i_1>n$ such that $i_1$ is adjacent to all vertices in the summation of $v_n^+$ while non-adjacent to all vertices in the summation of $v_n^-$. Therefore, we obtain
		\begin{align*}
			(Kf)_{i_1}=-\dfrac{1}{2} \left( v_n^+ +4 v_n^- +\sum_{j>n} v_j d^2(i_1,j)\right)
		\end{align*}
		Similarly we can find $i_2>n$ such that
		\begin{align*}
			(Kf)_{i_2}=-\dfrac{1}{2} \left( 4v_n^+ + v_n^- +\sum_{j>n} v_j d^2(i_2,j)\right)
		\end{align*}
		Using $(Kf)_{i_1}=(Kf)_{i_2}$, we get
		\begin{align}
			\label{eqn:vertex diff}
			-3v_n^+ +3 v_n^- +\sum_{j>n}(d^2(i_1,j)-d^2(i_2,j))v_j=0
		\end{align}
		Since the diameter of $(R,d)$ is $2$, we find
		\begin{align*}
			\biggr\vert \sum_{j>n} \left(d^2(i_1,j)-d^2(i_2,j)\right)v_j \biggr\vert \leq 4 \sum_{j>n} \vert v_j\vert
		\end{align*}
		Since $\lbrace v_j\rbrace$ is absolutely convergent, we find a contradiction to~\eqref{eqn:vertex diff} for sufficiently large $n$. It follows that $\ker T\cap \fone^{\perp}=\{ 0\}$. Since $\fone^{\perp}$ has codimension $1$ in $L^2(R,\mu)$, 
		the dimension of $\ker T$ cannot exceed $1$. Because $\ker T$ contains constants, they are
		the only elements of the kernel of $T$ as claimed.
	\end{proof}
	
	\begin{remark}
		In \cite{kroshnin2022infinite}, the authors showed for a homogeneous closed connected Riemannian manifold $X$ equipped with the canonical distance $d$ and Riemannian volume $\mu$, the constants are in the kernel of the MDS defining operator $T$ for $(X,d,\mu)$. The results on such smooth manifolds given by these authors appear to share similarities with the theorem above. However, in contrast to the Riemannian volume measure on homogeneous spaces, a probability measure with full support on a homogeneous countable metric space can never be homogeneous, i.e., invariant under automorphisms of the graph.
	\end{remark}

	\subsection{Signature ratio limit of i.i.d.\ generated finite subsets of the Rado graph}
	
	We examine the following question on the Rado graph. Let $(R_N)_{N\geq 1}$ be an increasing sequence of finite subsets of $R$ such that $R_N\nearrow R$, endowed the restriction of the metric $d$. Denote by $\Delta(R_N)$ the ratio of the number of positive eigenvalues over the number of negative eigenvalues of the associated matrix $S(R_N)$, that is,
	\begin{equation*}
		\Delta(R_N) \eqset \frac{s_+(S(R_N))}{s_-(S(R_N))}.
	\end{equation*}
	We ask whether the limit $\lim_{N\rightarrow \infty} \Delta(R_N)$ exists on the extended real line, and whether this limit is independent of the choice of $(R_N)_{N\geq 1}$. Note first that the Wigner semicircle law and the random construction of $R$ imply there exists a sequence $R_N\nearrow R$ such that $\lim_{N\rightarrow \infty} \Delta(R_N)=1$, see Section \ref{sec:mds op rado}. However, this is not true for all sequences, as the following example shows. 

	\begin{example}\label{ex_arRadoRatio1}
		Fix any point $R_1'=\{r_1\}$. By the finite extension property, if $R'_N\subset R$ is finite such that all pairs of vertices are adjacent, there exists $r_{N+1}\in R\setminus R_k'$ such that $R'_{N+1}=R'_{N}\cup \{r_{N+1} \}$ has all pairs of vertices adjacent. Thus, we obtain a strictly increasing sequence $(R'_N)_{N\geq 1}$ of finite subsets of $R$ such that the vertices of $R'_N$ are all adjacent for each $N\geq 1$. Each $R_N'$ with the metric induced by $d$ can be isometrically embedded into $\R^{N-1}$, but not in $\R^{N-2}$. By the 
		Proposition~\ref{prop_Hllbert1},
		we have $s_+(S(R_N'))=N-1$ and $s_-(S(R_N'))=1$ for this sequence. Note that $R\setminus\cup_{N=1}^{\infty} R'_N$ is infinite, i.e.\
		\[
		R\setminus\cup_{N=1}^{\infty} R'_N = \{t_1, t_2,\ldots \}.
		\]
		Define $R_N = R'_{N^2}\cup \{t_1,\ldots,t_N\}$, so we have $R_N\nearrow R$.
		using the Cauchy interlacing theorem, it is easy to show for the sequence $( R_N )_{N\geq 1}$ defined as above, we have
		$\lim_{N\rightarrow\infty} \Delta(R_N)=\infty$.
	\end{example}

	Thus we see the limit $\lim_{N\rightarrow\infty} \Delta(R_N)$ is not unique. Moreover, similarly one can show that it does not always exist on the extended real line.
	Although the limit of the ratios $\Delta(R_N)$ depends on the choice of $R_N\nearrow R$, they typically converge to $1$ in the following sense.
	
	\begin{proposition}
		\label{prop:random_rado_ratio}
		Let $G$ be a random Erd{\H o}s--R{\'e}nyi--Rado graph, $\mu$ be a probability measure on $\N$ with full support, and $\xi_i\in \N$,
		be i.i.d.\ random elements with $\operatorname{law}(\xi_i)=\mu$, $i\in \N$, independent of $G$.
		Then, for $X_m \eqset (\xi_i)_{i=1}^m$ one has
		\[
		\lim_{m \to \infty} \Delta(X_m) = 1
		\]
		almost surely.
	\end{proposition}
	
	\begin{proof}
		Since $\mu$ has an infinite support, a random sequence $\xi = (\xi_i)_{i=1}^\infty$ a.s.\ has infinitely many distinct elements. 
		For such a realization of $\xi$, define an auxiliary sequence $\hat\xi = (\hat \xi_i)_{i=1}^{\infty} \subset \N$ which is obtained from $\xi$ by cancelling the repeating elements in the order of their appearance.
		Since $G$ and $\xi$ are independent, the conditional law of $G$ given $\xi$ (uniquely defined for a.e.\ $\xi$) does not depend on $\xi$ and coincides with the marginal law of $G$, that is, the distributional law of an Erd{\H o}s--R{\'e}nyi--Rado random graph.
		The matrices $\hat{S}_N$ corresponding to $S_N$ (defined by~\eqref{eq_SNRado1}) are given by the formulae
		\[
		(\hat{S}_{N})_{ij} \eqset \dfrac{3}{2} \mathcal{A}_{\hat \xi_i, \hat \xi_j} - 2,\; (\hat{S}_{N})_{ii} = 0,\quad 1 \le i \neq j \le N.
		\]
		Since all $\hat \xi_i$ are by definition distinct, by the construction of an Erd{\H o}s--R{\'e}nyi--Rado graph and Proposition~\ref{theorem:semicircle law} we obtain that ESDs of $\hat{S}_{N}$ converge to the semicircle law a.s.\ (conditioned on $\xi$, i.e.\ with $\xi$ and hence also $\hat \xi$ fixed). Therefore,
		\[
		\lim_{N \to \infty} \frac{s_+(\hat{S}_N)}{s_-(\hat{S}_N)} = 1,
		\]
		and by the Fubini theorem, this holds (unconditionally) a.s.
		
		Finally, by the properties of the Rado graph, one a.s.\ has $(\hat{S}_{N})_{ij} = -\frac 1 2 d^2_G(\hat\xi_i, \hat \xi_j)$, i.e.\ $\hat{S}_{N} = S(\hat{X}_N)$, where $\hat{X}_N \eqset \{\hat \xi_i\}_{i=1}^N$. Then by Lemma~\ref{lm_cancelrep_matrix1} we get
		\[
		\lim_{m \to \infty} \Delta(X_m) = \lim_{N \to \infty} \Delta(\hat{X}_N) = \lim_{N \to \infty} \frac{s_+(S(\hat{X}_N))}{s_-(S(\hat{X}_N))} = \lim_{N \to \infty} \frac{s_+(\hat{S}_N)}{s_-(\hat{S}_N)} = 1 \text{ a.s.}
		\]
		as claimed.
	\end{proof}
	
	
	The above Proposition~\ref{prop:random_rado_ratio} shows that the limit of the ratios $\lim_{m\rightarrow\infty} \Delta\left((\xi_i)_{i=1}^m\right)$
	is equal to $1$ almost surely for almost every realization of the random Rado graph $G$, once the random points $\xi_i$ in the Gromov--Vershik scheme are independent of $G$. This is however not the case when we fix a specific realization of the Rado graph as the following examples show. In fact, we show that 
	the limit of the ratios $\lim_{m\rightarrow\infty} \Delta\left((\xi_i)_{i=1}^m\right)$ may become as large as desired (Example~\ref{ex_limratio_rado1}) or even infinity (Example~\ref{ex_Rado2}), all this of course, with some positive probability and depending on the choice of the measure $\mu$.
	
	\begin{example}\label{ex_limratio_rado1}
		Let $(R,d)$ be as before. Let $E \eqset \lbrace e_1, e_2,\ldots\rbrace \subset R$ be the infinite Hilbertian simplex, i.e.,\
		$d(e_i,e_j)=\delta_{ij}$ (we call it Hilbertian since $(E,d)$ can be isometrically embedded in the Hilbert space $\ell^2$).
		The set $E_0 \eqset R\setminus E$ is clearly infinite, and we enumerate its elements so that
		\[
		E_0 = \lbrace e_1^0,e_2^0,\ldots\rbrace.
		\] 
		Fix any integer $j\geq 1$. For $k=1,\ldots, j$, define
		\begin{align*}
			E_k\eqset \lbrace e^k_l\eqset e_{lj+k}: l\in \mathbb{N}\rbrace,
		\end{align*}
		so that $E=\bigcup_{k=1}^j E_k$, and the sets $E_k$, $k=1,\ldots, j$, are disjoint.
		We construct a probability measure $\mu$ of full support on $R$ such that $\mu(\{e^0_l\})=\mu(\{e^k_l\})$ 
		for all $l\geq 0$ and $k=1,\ldots, j$. Note that in this way
		\begin{equation}\label{eq_muVEk}
			\mu(E_0)=\mu(E_1)=\ldots= \mu(E_j) = \frac{1}{j+1}
		\end{equation}
		Let $\xi=(\xi_i)_{i\geq 1}$ with $\xi_i \colon \Omega\rightarrow R$ be the i.i.d.\ random sequence from $(R,d,\mu)$. 
		For each $k=0,\ldots, j$, denote by 
		$N_{E_k}(m)$ the number of \textit{distinct} elements in $E_k$ of the 
		finite random sequence $(\xi_1,\xi_2,\ldots,\xi_m)$.
		
		We show first that for $p_j\eqset 1/(j+1)>0$, one has
		\begin{equation}\label{eq_pj2}
			\P\left(\left\{N_{E_0}(m,\omega)\leq N_{E_k}(m,\omega), k=1,\ldots, j\right\}\right) \geq p_j
		\end{equation}
		for all $m\in \N$.
		In fact, setting
		\[
		\Omega_k\eqset \left\{\omega\in\Omega : N_{E_k}(m,\omega)\leq N_{E_i}(m,\omega), i=0,\ldots, j\right\},
		\]
		we clearly get
		\[
		\P(\Omega_0)= \P(\Omega_1)= \ldots =\P(\Omega_j),
		\]
		and $\bigcup_{k=0}^j \Omega_k = \Omega$.
		Therefore, 
		\[
		1= \P(\Omega)\leq \sum_{k=0}^j \P(\Omega_k)= (j+1) \P(\Omega_0),
		\]
		hence the claim~\eqref{eq_pj2}.

		Once~\eqref{eq_pj2} is proven, we consider the set $\hat {R}_{N(m)}$ of distinct elements
		of the finite random sequence $(\xi_i)_{i=1}^m$, and the respective square matrix
		$S({\hat {R}_{N(m)}})$ as well as the square matrices 
		\begin{align*}
			\hat S_{N_{(E_0,m)}}^0 &\eqset -\frac{1}{2} (d^2(\xi_i, \xi_j))_{i,j=1,\ldots, m, \xi_i, \xi_j\in E_0\cap \hat {R}_{N(m)}},\\
			\hat E_{N_E(m)} &\eqset -\frac{1}{2} (d^2(\xi_i, \xi_j))_{i,j=1,\ldots, m, \xi_i, \xi_j\in E\cap \hat {R}_{N(m)}},
		\end{align*}
		where the lower index stands for the size of the respective matrix.
		One has
		\begin{align*}
			s_+\left(E_{N_E(m)}\right) &= N_E(m)- 1, 
		\end{align*}
		since $E_{N_E(m)}$ is $N_E(m)\times N_E(m)$ matrix and $(E,d)$ is isometrically embedded in a Hilbert space.
		Hence 
		\begin{equation}\label{eq_sEm1}
			s_+\left(\hat{E}_{N_E(m)}\right)= N_E(m)-1.
		\end{equation}
		We observe now that the matrix $S({\hat {R}_{N(m)}})$ has the block form 
		\begin{align*}
			S({\hat {R}_{N(m)}})=
			\left(\begin{matrix}
				\hat S_{N_{E_0}(m)}^0 && *\\
				* && \hat E_{N_E(m)} 
			\end{matrix}
			\right)
		\end{align*}
		(where $*$ stands for some blocks).
		Thus in view of the Cauchy interlacing theorem, we get
		\begin{equation}\label{eq_sXm1}
			s_+(S({\hat {R}_{N(m)}}))\geq s_+(\hat E_{N_E(m)} ) = N_E(m)-1, 
		\end{equation}
		the latter equality due to~\eqref{eq_sEm1}.
		On the other hand, $S({\hat {R}_{N(m)}})$ is an $N(m)\times N(m)$ matrix, and hence
		\begin{equation}\label{eq_sXm2}
			\begin{aligned}
				s_-(S({\hat {R}_{N(m)}})) &\leq N(m)- 	s_+(S({\hat {R}_{N(m)}}))\\
				& \leq N(m)-(N_E(m)-1) \qquad \mbox{by~\eqref {eq_sXm1}}\\
				& = N_{E_0}(m)+1.
			\end{aligned}
		\end{equation}
		From~\eqref{eq_sEm1} by~\eqref{eq_pj2}, of probability at least $p_j>0$, one has
		\begin{equation}\label{eq_sEm1b}
			s_+\left(\hat{E}_{N_E(m)}\right)=N_E(m)-1=\sum_{k=1}^j N_{E_k}(m)-1\geq jN_{E_0}(m)-1
		\end{equation}
		for sufficiently large $m$. Hence, combining~\eqref{eq_sXm2},~\eqref{eq_sEm1b}, and~\eqref{eq_sXm1}, we arrive at
		\begin{align}
			\Delta(\hat {R}_{N(m)})\geq \dfrac{ j N_{E_0}(m)-1}{ N_{E_0}(m)+1} \geq \dfrac{j}{3}.
		\end{align}
		Hence, for any $j\geq 1$, we have
		\begin{align*}
			\liminf_{m\rightarrow \infty} \P\left(\left\{\Delta(\hat{R}_{N(m)}) \geq \dfrac{j}{3}\right\}\right) \geq p_j>0.
		\end{align*}
		Setting $j$ sufficiently large, we may have $\lim_{m\rightarrow\infty} \Delta\left((\xi_i)_{i=1}^m\right)$ as large as desired with a positive probability.
	\end{example}

	\begin{example}\label{ex_Rado2}
		Let $(R,d)$ be as before. 
		According to the Example~\ref{ex_arRadoRatio1}, there exists an enumeration $R = \{r_i\}_{i=1}^\infty$ such that $\Delta(R_N) \to \infty$ with $R_N \eqset \{r_1, \dots, r_N\}$.
		For the sake of brevity, let us identify $R$ with $\N$ according to this enumeration.
		Define a probability measure $\mu$ on $\N$ by $\mu(\{k\}) \eqset C 2^{-k^2}$ with the normalizing constant $C \eqset \left(\sum_{k=1}^\infty 2^{-k^2}\right)^{-1}$.
		As usual, let $(\xi_i)_{i=1}^\infty$ be a sequence of i.i.d.\ random elements of $R$ with law $\mu$ and $(\hat{\xi}_j)_{j=1}^\infty$ be the corresponding sequence obtained from $\xi$ by cancelling the repeating elements in the order of their appearance.
		It is easy to see that $\hat{\xi}_i = i$ for all $i \in \N$ if and only if $\xi_j \le \max_{1 \le i < j} \xi_i + 1$ for all $j \in \N$, with a convention $\max \emptyset = 0$.
		Furthermore,
		\begin{equation}\label{eq_Pest1}
			\begin{aligned}
				\P\left(\left\{\exists j \in \N : \xi_j > \max_{1 \le i < j} \xi_i + 1\right\}\right)
				&\le \sum_{j=1}^\infty \P\left(\left\{\xi_j > \max_{1 \le i < j} \xi_i + 1\right\}\right) \\
				&= \sum_{j=1}^\infty \sum_{k = 0}^\infty \P\left(\left\{\xi_j = k + 2, \, \max_{1 \le i < j} \xi_i \le k\right\}\right) .
			\end{aligned}
		\end{equation}	
		Since all $\xi_i$ are i.i.d., we have (with convention $0^0 = 1$ in the case $k = 0$, $j = 1$)
		\[
		\P\left(\left\{\xi_j = k + 2, \, \max_{1 \le i < j} \xi_i \le k\right\}\right)
		= \P(\{\xi_1 = k + 2\}) \bigl(\P(\{\xi_1 \le k\})\bigr)^{j-1} .
		\]
		Note that 
		\[
		\P(\{\xi_1 \le k\}) = 1 - \P(\{\xi_1 \ge k+1\}) \le 1 - \P(\{\xi_1 = k+1\}) = 1 - \mu(\{k+1\}). 
		\]
		Thus, we continue the chain of estimates~\eqref{eq_Pest1} by
		\begin{align*}
			\P\left(\left\{\exists j \in \N : \xi_j > \max_{1 \le i < j} \xi_i + 1\right\}\right)
			&\le \sum_{j=1}^\infty \sum_{k = 0}^\infty \mu(\{k+2\}) \bigl(1 - \mu(\{k+1\})\bigr)^{j-1} \\
			&= \sum_{k = 0}^\infty \mu(\{k+2\}) \sum_{j=1}^\infty \bigl(1 - \mu(\{k+1\})\bigr)^{j-1} \\
			&= \sum_{k = 0}^\infty \frac{\mu(\{k+2\})}{\mu(\{k+1\})} = \sum_{k = 0}^\infty 2^{(k+1)^2 - (k+2)^2} \\
			&= \sum_{k = 0}^\infty 2^{-2 k - 3} = \frac{1}{6} .
		\end{align*}
		Therefore, with probability at least $\frac{5}{6}$, one has $\xi_j \le \max_{1 \le i < j} \xi_i + 1$ for all $j \in \N$, hence $\hat{\xi}_i \equiv i$, and thus
		\[
		\lim_{m\to\infty} \Delta\bigl((\xi_i)_{i=1}^m\bigr) = \lim_{N\to\infty} \Delta\bigl((\hat{\xi}_i)_{i=1}^N\bigr) 
		= \lim_{N\to\infty} \Delta(R_N) = \infty .
		\]
	\end{example}

	\appendix
	\section{Auxiliary lemmas on operators and their spectra}
	
	We collect here some technical results on the spectra of operators on Hilbert spaces used throughout the paper.
	
	\begin{lemma}\label{lem:sign_comparison}
		Let $H$ be a Hilbert space.
		\begin{itemize}
			\item[(i)] If $A$, $B$ are compact self-adjoint operators on $H$, then
			\begin{equation*}
				s_\pm(A + B) \le s_\pm(A) + s_\pm(B) .
			\end{equation*}
			\item[(ii)] If $A$ and $\{A_n\}_{n \ge 1}$, are compact self-adjoint operators on $H$ such that $A_n \to A$ in operator norm, then
			\begin{equation*}
				s_\pm(A) \le \liminf s_\pm(A_n) .
			\end{equation*}
		\end{itemize}
	\end{lemma}
	
	\begin{proof}
		We prove the claims for $s_+$. The proof for $s_-$ is identical.
		
		We start with the proof of (i). Assume that $s_+(A)<\infty$ and $s_+(B) < \infty$, otherwise the claim is trivial. Suppose that $s_+(A + B) > s_+(A) + s_+(B)$. 
		Then there are at least $m+1 = s_+(A) + s_+(B) + 1$ linearly independent eigenvectors $u_1, \dots, u_m$ of $A + B$ corresponding to positive eigenvalues.
		Let $L_A$ (resp., $L_B$) be the linear span of all eigenvectors of $A$ (resp.,\ $B$) corresponding to positive eigenvalues. Clearly, we have that $\dim L_A = s_+(A)$ and $\dim L_B = s_+(B)$. 
		Then $\dim(L_A + L_B) \le s_A(A) + s_+(B) < m$ and there is a non-zero vector $u \in \Span \{u_i\}_{i=1}^m \cap (L_A + L_B)^\perp$. Since $A$ and $B$ are self-adjoint, this yields $(u, A u) \leq 0$, $(u, B u) \le 0$, hence $(u, (A + B) u) \le 0$ what leads to a contradiction. Therefore, $s_+(A + B) \le s_+(A) + s_+(B)$ as claimed.

		Now, let us prove (ii). Let $u_1, \dots, u_m$ be unit norm eigenvectors of $A$ corresponding to positive eigenvalues $\lambda_1 \ge \lambda_2 \ge \dots \ge \lambda_m > 0$. Define symmetric matrices $\bar{A} \eqset ((u_i, A u_j))_{i,j=1}^m = \diag(\lambda_1, \dots, \lambda_m)$ and $\bar{A}_n \eqset ((u_i, A_n u_j))_{i,j=1}^m$. Then for $\lVert A_n - A\rVert < \lambda_m$, we have
		\[
		\lambda_{\min}(\bar{A}_n) \ge \lambda_{\min}(\bar{A}) + \lambda_{\min}(\bar{A}_n - \bar{A}) 
		\ge \lambda_m - \lVert\bar{A}_n - \bar{A}\rVert \ge \lambda_m - \lVert A_n - A\rVert > 0 .
		\]
		In this case, $s_+(A_n) \ge s_+(\bar{A}_n) \ge m$, therefore, 
		\[
		m \le \liminf_n s_+(A_n).
		\]
		Taking supremum over all $m\leq s_+(A)$, we get the result. 
	\end{proof}

	\begin{lemma}\label{lm_lamkT2a}
		Let $(X,d)$ be a metric space, $\mu$ be a Borel probability measure, and
		$A\colon L^2(X,\mu)\to L^2(X,\mu)$ be an integral linear compact operator defined by the formula
		\[
		(Au)(x) \eqset \int_X a(x,y) u(y)\, d\mu(y)
		\] 
		with $a\in L^2(X\times X, \mu\otimes \mu)$ satisfying $a(x,y)=a(y,x)$ for $ \mu\otimes \mu$-a.e.\ $(x,y)\in X\times X$, so that
		$A$ is self-adjoint. 
		Let $\lambda_i\in \R$ and $u_i\in L^2(X,\mu)$ be an enumeration of its eigenvalues (counting multiplicities) and the respective eigenfunctions of unit norm in $L^2(X,\mu)$. 
		Then 
		\begin{equation}\label{eq_lamkT2a}
			a(x,y) = \sum_i \lambda_i u_i(x) u_i(y) ,
		\end{equation}
		the equality in the sense of $L^2(X\times X,\mu\otimes\mu)$.
	\end{lemma}
	
	\begin{proof}
		By the Hilbert--Schmidt theorem, $\{u_i\}$ form a topological basis in $L^2(X,\mu)$. Thus 
		representing $a$ in terms of the Fourier series in $u_i\otimes u_j$ in $L^2(X\times X,\mu\otimes\mu)$,
		we get
		\[
		a(x,y)= \sum_{i,j} c_{ij} u_i(x) u_i(y).
		\]
		For the Fourier coefficients $c_{ij}$, we have
		\begin{align*}
			c_{ij}&=\int_X \left(\int_X a(x,y)u_i(u)\,d\mu(y) \right)u_j(x)\,d\mu(x) = \int_X \lambda_i u_i(x)u_j(x)\,d\mu(x)\\
			&= \lambda_i \delta_{ij},
		\end{align*}
		where $\delta_{ij}$ stands for the Kronecker delta, which implies~\eqref{eq_lamkT2a}.
	\end{proof}
	
	The following lemmas have been used in the proof of Proposition~\ref{prop:sign convergence2_sep}.
	
	\begin{lemma}\label{lm_approxLinOp1}
		Let $(X,d)$ be a metric space, $\mu$ be a Borel probability measure, and
		$A\colon L^2(X,\mu)\to L^2(X,\mu)$ be an integral linear operator defined by the formula
		\[
		(Au)(x) \eqset \int_X a(x,y) u(y)\, d\mu(y)
		\] 
		with a continuous function $a\in L^2(X\times X, \mu\otimes \mu)$ satisfying $a(x,y)=a(y,x)$. The following assertions hold. 
		\begin{itemize}
			\item[(i)] For every finite subset $X_{N} = \{x_1,\ldots,x_N\} \subset \supp \mu$, one has
			\begin{eqnarray}
				\label{eq_spmA1}
				s_\pm(A) \geq s_\pm(A(X_N)), \\
				\label{eq_spmP1}
				s_\pm(P_\mu A P_\mu) \geq s_\pm(\Pi_N A(X_N) \Pi_N),
			\end{eqnarray}	
			where $A(X_N)\eqset (a(x_i, x_j))_{i,j=1}^N \in \R^{N \times N}$.
			\item[(ii)] If $\mu$ is a Radon measure (which is automatically guaranteed when $(X,d)$ is Polish), then there is a sequence of finite subsets $X_N\subset X$ such that
			\begin{eqnarray}
				\label{eq_spmA2}		
				s_\pm (A) \leq \liminf_N s_\pm (A(X_N)), \\\
				\label{eq_spmP2}
				s_\pm(P_\mu A P_\mu) \leq \liminf_N s_\pm(\Pi_N A(X_N) \Pi_N).
			\end{eqnarray}
		\end{itemize}
	\end{lemma}
	
	\begin{proof}
		To prove~(i), let $\varepsilon>0$ be such that all the positive eigenvalues of $A(X_N)$ belong to the interval $[\epsilon,+\infty)$.
		Choose an $r>0$ sufficiently small such that all open balls $B_i \eqset B_{r}(x_i)$, $i = 1, \dots, N$, are pairwise disjoint, and for all $i,j=1,\ldots, N$ the estimate
		\begin{align}
			\label{eqn:component_perturbation}
			\left| a(x_i',x_j') - a(x_i,x_j)\right| \le \frac{\varepsilon}{2},\quad\mbox{ for all $x_i'\in B_i,\ x_j'\in B_j$} .
		\end{align}
		holds. Define
		\[
		v_i \eqset \frac{1}{\mu(B_i)} \fone_{B_i} \in L^2(X,\mu), \qquad i = 1, \dots, N, 
		\]
		(recall that $\mu(B_i) > 0$ since $x_i\in \supp \mu$).
		Let $Q_N \colon L^2(X,\mu)\to L^2(X,\mu)$ be the orthogonal projector onto $L_N \eqset \Span \{v_i\}_{i=1}^N$.
		Define the linear operator $\tilde{A}_N \colon L^2(X,\mu) \to L^2(X,\mu)$ by
		\[
		(\tilde{A}_N u)(x) \eqset \sum_{i,j=1}^N \int_X a(x_i, x_j) \fone_{B_i}(x) \fone_{B_j}(y) u(y)\, d\mu(y) .
		\]
		For every $u\in L_N$, setting $U\eqset \bigcup_{i=1}^N B_i$, we have
		\begin{equation}\label{eq_estAN1}
			\|\fone_U A u - \tilde{A}_N u\|_2 \leq \|\fone_U A u - \tilde{A}_N u\|_\infty 
			\leq \frac{\varepsilon}{2} \|u\|_1 \leq \frac{\varepsilon}{2} \|u\|_2, 
		\end{equation}
		where the first and the last estimates are just H\"{o}der inequality minding that $\mu$ is a probability measure, and the second one is due to~\eqref{eqn:component_perturbation}.
		Since $Q_N v = Q_N (\fone_U v)$ for all $v \in L^2(X,\mu)$, then from~\eqref{eq_estAN1} we get that
		the operator norm of $Q_N A Q_N - \tilde{A}_N = Q_N (A - \tilde{A}_N) Q_N$ is bounded by $\frac{\varepsilon}{2}$.
		Now, let us define the embedding operator $E_N \colon \R^N \to L_N$ as $E_N x \eqset \sum_{i=1}^N x_i v_i$, so that $\Img E_N = L_N$ and 
		\[
		(E_N^* v_i)_k= \frac{1}{\mu(B_i)} \delta_{ik} 
		\]
		with $\delta_{ik}$ standing for the Kronecker delta.
		Since
		\begin{align*}
			\tilde{A}_N(v_k)(x)&= \sum_{i=1}^N \fone_{B_i}(x) \sum_{j=1}^N \int_X a(x_i,x_j) 1_{B_j}(y) \dfrac{1}{\mu(B_k)} \fone_{B_k}(y) d\mu(y),\\
			& = \sum_{i=1}^N \fone_{B_i}(x) a(x_i,x_k) = \sum_{i=1}^N a(x_i,x_k) \mu(B_i) v_i,
		\end{align*}
		then
		$E_N^* \tilde{A}_N E_N = A(X_N)$. 
		We have therefore
		\begin{align*}
			s_+(A) &\geq s_+(Q_N A Q_N) \\
			&\qquad \mbox{by Lemma~\ref{lm spm_control1}(i) with $Q\eqset Q_N$, $B\eqset Q_N A Q_N$, $H_1=H_2=L^2(X,\mu)$}\\
			&\geq s_+(\tilde{A}_N) \quad\mbox{by the Weyl inequality}\\
			&= s_+(E_N^* \tilde{A}_N E_N) \quad\mbox{by Lemma~\ref{lm spm_control1}(iii) since $(\ker \tilde{A}_N)^\perp \subset L_N = \Img E_N$}\\
			&= s_+(A(X_N)).
		\end{align*}
		
		Now, let $\bar{Q}_N \colon L^2(X,\mu)\to L^2(X,\mu)$ be the orthogonal projector onto $L_N \cap \fone^\perp$.
		In the same way as above, $\bar{Q}_N (A - \tilde{A}_N) \bar{Q}_N$ is bounded by $\frac{\varepsilon}{2}$.
		Furthermore,
		\[
		(\fone, E_N x) = \sum_{i=1}^N (\fone, v_i) x_i = (\vone_N, x) ,
		\]
		hence $\Img(E_N \Pi_N) = L_N \cap \fone^\perp$ and $\bar{Q}_N E_N \Pi_N = E_N \Pi_N$. Therefore,
		\begin{align*}
			s_+(P_\mu A P_\mu) &\geq s_+(\bar{Q}_N A \bar{Q}_N) \quad\mbox{by Lemma~\ref{lm spm_control1} since $P_\mu \bar{Q}_N = \bar{Q}_N$}\\
			& \geq s_+(\bar{Q}_N \tilde{A}_N \bar{Q}_N) \quad\mbox{by the Weyl inequality}\\
			& \ge s_+(\Pi_N E_N^* \tilde{A}_N E_N \Pi_N) \quad\mbox{by Lemma~\ref{lm spm_control1} since $\bar{Q}_N E_N \Pi_N = E_N \Pi_N$}\\
			&= s_+(\Pi_N A(X_N) \Pi_N).
		\end{align*}
		
		The proofs of $s_-(A) \geq s_- (A(X_N))$ and $s_-(P_\mu A P_\mu) \geq s_-(\Pi_N A(X_N) \Pi_N)$ are completely symmetric.
		
		%
		
		We now prove~(ii). To this aim, for an arbitrary $\varepsilon>0$, using that $\mu$ is the Radon measure,
		we can find a compact $K\subset X$ with $\mu(X\setminus K)$ be so small that 
		\[
		\int_{(X\times X) \setminus (K\times K)} a^2(x,y)\, d\mu(x)d\mu(y) <\varepsilon^2.
		\]
		This can be done in view of by absolute continuity of the integral and of the fact that
		\begin{align*}
			\mu((X\times X) \setminus (K\times K)) &= \mu(((X\setminus K)\times X) \cup (K\times (X\setminus K))) \\
			&\leq 2 \mu(X\setminus K).
		\end{align*}
		We find a $\delta>0$, a $\delta$-net $X_N\eqset \{x_1, \ldots, x_N\} \subset K$, and a partition of $K$ into $N$ disjoint Borel subsets $B_1,\ldots, B_N$ such that
		\begin{align}
			\label{eqn:component_perturbation1}
			\left| a(x_i',x_j')-a(x_i,x_j)\right|<\varepsilon,\quad\mbox{ for all $x_i'\in B_i$, $x_j'\in B_j$}
		\end{align}
		(e.g., $B_i$ being a ``Voronoi cell'' for $x_i$).
		
		Much similar to the prof of~(i), let $Q_N\colon L^2(X,\mu)\to L^2(X,\mu)$ be the orthogonal projection onto $\Span \{\fone_{B_i}\}_{i=1}^N$ and denote by $\tilde A(X_N)\colon L^2(X,\mu)\to L^2(X,\mu)$ the linear operator defined by the formula
		\[
		(\tilde{A}_N u)(x) \eqset \sum_{i,j=1}^N \int_X a(x_i,x_j) \fone_{B_i}(x) \fone_{B_j}(y) u(y)\, d\mu(y).
		\] 
		For every $u\in L^2(X,\mu)$, we have
		\begin{align*}
			\|Q_N A Q_N u-\tilde{A}_N u\|_2 &
			\leq 
			\|Q_N A Q_N u-\tilde{A}_N u\|_\infty \leq \varepsilon \|u\|_1 \leq \varepsilon \|u\|_2, 
		\end{align*}
		and
		\begin{align*}
			\|Au - Q_N A Q_N u\|_2^2 &
			\leq \int_X \, d\mu(x) \left(\int_X \left| a(x,y)- \fone_K(x) a(x,y)\fone_K(y)\right|\cdot |u(y)|\, d\mu(y) \right)^2
			\\
			&\leq \|u\|_2^2 \int_X \, d\mu(x) \left(\int_X \left| a(x,y) - \fone_K(x) a(x,y)\fone_K(y)\right|^2\, d\mu(y) 
			\right)\\
			&\leq \varepsilon^2 \|u\|_2^2 . 
		\end{align*}
		This implies
		\[
		\|Au - \tilde{A}_N u\|_2 \leq \|A u - Q_N A Q_N u\|_2 + \|Q_N A Q_N u - \tilde{A}_N u\|_2 
		\le 2 \varepsilon \|u\|_2.
		\]
		In other words, we constructed a sequence of $X_N\subset X$ such that
		$\tilde{A}_N \to A$ as linear bounded operators over $L^2(X,\mu)$ as $N\to\infty$.
		This implies 
		\begin{equation*}\label{eq_smptildAN1}
			s_\pm (A)\leq \liminf_N s_\pm (\tilde{A}_N) 
		\end{equation*}
		by Lemma~\ref{lem:sign_comparison}(ii).
		In the same way one can show that
		\[
		s_\pm (P_\mu \tilde{A}_N P_\mu)\leq \liminf_N s_\pm (\Pi_N A(X_N) \Pi_N),
		\]
		hence concluding the proof.
	\end{proof}

	\begin{lemma}\label{lm spm_control1}
		Let $H_1$, $H_2$ be Hilbert spaces, $A$ and $B$ be linear compact self-adjoint operators over $H_2$ and $H_1$, respectively, such that
		$B = Q^* A Q$, where $Q\colon H_1 \to H_2$ is a linear bounded operator.
		Then
		\begin{itemize}
			\item[(i)] $s_\pm(B) \le s_\pm(A)$. 
			\item[(ii)] If, moreover, the codimension of the image of $Q$ is $k\in \N$, then
			\[
			s_\pm(B) \geq s_\pm(A) - k.
			\]	
			\item[(iii)]
			If $(\Img Q)^\perp \cap (\ker A)^\perp = \{0\}$, then $s_\pm(A) = s_\pm(B)$.
		\end{itemize}
	\end{lemma}
	
	\begin{proof}
		Let $\{u_j\}_{j=1}^n\subset H_1$ be a set of linearly independent eigenvectors of $B$ 
		corresponding to strictly positive eigenvalues. 
		The vectors $Q u_j$, $j=1,\ldots, n$ are also linearly independent
		since otherwise $\ker B \cap \Span \{u_j\}_{j=1}^n\neq \{0\}$ which is impossible.
		Hence, $(Av, v)>0$ for every $v\in \Span \{Q u_j\}_{j=1}^n$, $v\neq 0$. 
		This implies $s_+(A)\geq n$, and therefore, taking the supremum over all $n\in \N$ such that $n\leq s_+(B)$, we get $s_\pm(B) \le s_\pm(A)$, proving Claim~(i).
		
		We now prove~(ii). Let $\{w_j\}_{j=1}^m \subset H_2$ be a set of linearly independent eigenvectors of $A$ corresponding to strictly positive eigenvalues.
		Then 
		\[
		\dim \left(\Span \{w_j\}_{j=1}^m\cap \Img Q \right) \geq m - k.
		\]
		Thus we may choose at least $m-k$ linearly independent elements 
		\[
		\{v_j\}_{j=1}^{m-k} \subset \Span \{w_j\}_{j=1}^m \cap \Img Q
		\]
		Pick then linearly independent $u_j\in H_1$ with $Q u_j = v_j$, $j=1,\ldots, m-k$.
		One has 
		\[
		(u, B u) = (Q u, A Q u) > 0
		\]
		for all 
		$u\in \Span \{u_j\}_{j=1}^{m-k}$ unless $u=0$. Thus implies $s_+(B)\geq m-k$ and taking the supremum over all $m\in \N$ satisfying $m\leq s_+(A)$, we have $s_\pm(B) \geq s_\pm(A)-k$, i.e.,\ Claim~(ii).
		
		Finally, to prove Claim~(iii), we first note that for the orthogonal projector $P_A$ from $H_2$ onto $(\ker A)^\perp$ one has
		$P_A A = A$. Hence, $A = A^* = A^* P_A^* = A P_A$, and this again implies 
		\[
		A = P_A A = P_A A P_A. 
		\]
		Therefore,
		\begin{equation}\label{eq_BPAQ1}	
			B = Q^* A Q = Q^* P_A A P_A Q = Q^* P_A^* A P_A Q = (P_A Q)^* A (P_AQ).
		\end{equation}
		Furthermore,
		\[
		(\Img Q)^\perp \cap (\ker A)^\perp = (\Img Q + \ker A)^\perp = \{0\},
		\]
		hence $\Img Q + \ker A$ is dense in $H_2$, thus $\Img(P_A Q)$ is dense in $\Img P_A = (\ker A)^\perp$. 
		Let again $\{w_j\}_{j=1}^m\subset H_2$ be a set of linearly independent eigenvectors of $A$ 
		corresponding to strictly positive eigenvalues, thus the matrix $(A w_i, w_j)_{i,j=1}^m$ is positive definite. 
		Since all $w_j$ clearly belong to $(\ker A)^\perp$, then one can find $\{\tilde w_j\}_{j=1}^m \subset \Img(P_A Q)$ with each $\tilde w_j$ so close to $w_j$ that
		they are all still linearly independent and the matrix $(A \tilde w_i, \tilde w_j)_{i,j=1}^m$ is positive definite (which gives in particular $(A\tilde w,\tilde w)>0$ for all $\tilde w\in \Span \{\tilde w_j\}_{j=1}^m$, $\tilde w\neq 0$). 
		
		Choosing $u_j \in H$ linearly independent such that
		$P_A Q u_j = \tilde w_j$, in view of~\eqref{eq_BPAQ1} we get
		\[
		(u, B u) = ((P_A Q) u, A (P_A Q) u) > 0
		\]
		for all $u\in \Span \{u_j\}_{j=1}^m$, $u\neq 0$. 
		Thus, $s_+(B)\geq m$ and taking the supremum over all $m\in \N$ satisfying $m\leq s_+(A)$, we have $s_\pm(B) \geq s_\pm(A)$. 
		This together with Claim~(i) gives Claim~(iii).
	\end{proof}
	
	\begin{lemma}\label{lem:strict_triangle}
		Let $\Sigma = \{z_i\}_{i=1}^3 \subset \R^{n,p}$ be such that $\Sigma - \Sigma \subset C_{n,p}$ (thus, $d_{n,p}(z_i, z_j)$ is defined).
		Then the strict triangle inequality 
		\[
		d_{n,p}(z_1, z_3) < d_{n,p}(z_1, z_2) + d_{n,p}(z_2, z_3)
		\]
		is equivalent to the strict one-sided Cauchy-Schwarz inequality 
		\[
		(z_1 - z_2, z_2 - z_3)_{n,p} < d_{n,p}(z_1, z_2) d_{n,p}(z_2, z_3) .
		\]
	\end{lemma}
	
	\begin{proof}
		The strict triangle inequality is clearly equivalent to
		\[
		d_{n,p}^2(z_1, z_3) < d_{n,p}^2(z_1, z_2) + 2 d_{n,p}(z_1, z_2) d_{n,p}(z_2, z_3) + d_{n,p}^2(z_2, z_3) .
		\]
		Since 
		\begin{align*}
		d_{n,p}^2(z_1, z_3) &\eqset (z_1 - z_3, z_1 - z_3) \\
		&= (z_1 - z_2, z_1 - z_2) + 2 (z_1 - z_2, z_2 - z_3) + (z_2 - z_3, z_2 - z_3) \\
		&= d_{n,p}^2(z_1, z_2) + 2 (z_1 - z_2, z_2 - z_3) + d_{n,p}^2(z_2, z_3) ,
		\end{align*}
		we get the result.
	\end{proof}

	\bibliographystyle{abbrv}
	\bibliography{countable-space-bib}

\end{document}